\documentclass{amsart}

\usepackage{amsmath,amssymb}
\usepackage{amsfonts}
\usepackage{graphicx}
\usepackage{latexsym}
\usepackage{array}
\usepackage{subfig}
\usepackage{stmaryrd}
\usepackage{color}
\usepackage{multicol}
\usepackage{vwcol}
\usepackage{xcolor}

\usepackage[T3,T1]{fontenc}
\DeclareSymbolFont{tipa}{T3}{cmr}{m}{n}
\DeclareMathAccent{\invbreve}{\mathalpha}{tipa}{16}
\usepackage[
    colorlinks=true,
    linkcolor=blue,      
    citecolor=blue,      
    urlcolor=blue        
]{hyperref}

\numberwithin{equation}{section} \numberwithin{figure}{section}

{\theoremstyle{remark} }
\newtheorem{definition}{Definition}[section]

\newtheorem{theorem}{Theorem}[section]
\newtheorem{corollary}{Corollary}[section]
\newtheorem{lemma}{Lemma}[section]

\def\XXint#1#2#3{{\setbox0=\hbox{$#1{#2#3}{\int}$ }
\vcenter{\hbox{$#2#3$ }}\kern-.6\wd0}}

\DeclareMathOperator{\supp}{supp}

\newcommand{\RR}{{\mathbb{R}}}
\newcommand{\NN}{{\mathbb{N}}}
\newcommand{\eps}{\varepsilon}
\newcommand{\mE}{\mathcal{E}}
\newcommand{\mc}[1]{\mathcal{#1}}
\DeclareMathOperator{\diam}{diam}
\DeclareMathOperator{\dist}{dist}

\begin{document} 
\title[Almost $L^\infty$-stability of the Ritz projection for $0<s\leq 1/2$]{Almost $L^\infty$-stability of the Ritz projection for the integral fractional Laplacian with $0<s\leq \tfrac12$}
\thanks{This work was supported by the National Natural Science Foundation
of China under Grant Nos.~12571383 and~12288101.}
\author[J.~Li]{Jia Li}
\address[J.~Li]{School of Mathematical Sciences, Peking University,
  Beijing, 100871, P.R. China}
\email{lijia2002@stu.pku.edu.cn}

\author[S.~Wu]{Shuonan~Wu}
\address[S.~Wu]{School of Mathematical Sciences, Peking University,
  Beijing, 100871, P.R. China}
\email{snwu@math.pku.edu.cn}


\date{}

\begin{abstract}
For $0<s\leq1/2$ and $d\geq1$, the $\widetilde H^s$-projection
onto piecewise linear finite element spaces on quasi-uniform meshes of a
bounded Lipschitz domain in $\RR^d$ with a uniform exterior ball
condition is almost stable in the maximum norm.
The mesh domain need not coincide with the original domain.
\end{abstract}

\maketitle


\section{Introduction} \label{sec:intro}

This paper concerns the maximum-norm stability of the finite element Ritz
projection associated with the integral fractional Laplacian. We consider
the homogeneous Dirichlet problem
\begin{equation}\label{eqn:intro-model}
 \begin{cases}
  (-\Delta)^s u=f & \text{in }\Omega,\\
  u=0 & \text{in }\Omega^c:=\RR^d\setminus\Omega,
 \end{cases}
\end{equation}
where $0<s\leq\frac{1}{2}$ and $\Omega\subset\RR^d$, $d\geq1$, is a bounded Lipschitz domain satisfying
a uniform exterior ball condition. The natural energy space is
$\widetilde H^s(\Omega)$, consisting of functions whose zero extensions
belong to $H^s(\RR^d)$. If $a_s(\cdot,\cdot)$ denotes the associated energy
form and $V_h$ is a conforming finite element space, the Ritz projection
$R_h^s$ is determined by
\[
 a_s(R_h^s u,v_h)=a_s(u,v_h)\qquad\forall v_h\in V_h.
\]

Maximum-norm estimates play an important role in the finite element
analysis of nonlinear and parabolic problems~\cite{Thomee2006}, as well as
optimal control problems with pointwise state constraints. We refer
to~\cite{Wahlbin1991} for the underlying local and pointwise error
analysis. For the integral fractional
Laplacian, boundary singularities occur even for smooth
data~\cite{RosOtonSerra2014,AcostaBorthagaray2017}, limiting
the applicability of error estimates that require globally bounded
higher derivatives. A maximum-norm best approximation estimate adapts
to the regularity of the solution through its approximation in the finite
element space. This motivates the study of maximum-norm stability of the
Ritz projection for bounded functions in the energy space.

For the classical Poisson equation, maximum-norm estimates for finite
element approximations have a long history. Classical developments
include weighted-norm estimates~\cite{Nitsche1975} and Green's function
analysis~\cite{Scott1976}. Our analysis follows the local approach of
Schatz and Wahlbin~\cite{SchatzWahlbin1977}, combining local energy
error estimates~\cite{NitscheSchatz1974,DemlowGuzmanSchatz2011}, dyadic
decompositions, and localized duality. This approach yields maximum-norm
almost best approximation for the $H_0^1$-projection on quasi-uniform
meshes over smooth domains~\cite{SchatzWahlbin1982}. Corresponding
results on convex polygons and polyhedra were obtained
in~\cite{LeykekhmanVexler2016}. Further results can be found
in~\cite{SchatzWahlbin1978,SchatzWahlbin1979,Schatz1980,
ApelRoschSirch2009,ApelRogovsPfeffererWinkler2020,Li2022}
and the references therein.

For the integral fractional Laplacian, a priori energy-norm error estimates
for conforming finite elements on quasi-uniform and graded meshes were
established in~\cite{AcostaBorthagaray2017}, while $H^1$-norm error estimates
for $s>1/2$ were obtained in~\cite{BorthagarayCiarlet2019}.
Adaptive algorithms and a posteriori error analysis have also been
developed for this and related nonlocal
problems~\cite{AinsworthGlusa2017,FaustmannMelenkPraetorius2021,
GimperleinStocek2019,NochettoPetersdorffZhang2010,ZhaoHuCaiKarniadakis2017}.
Other approaches include a nonconforming discretization based on a
Dunford--Taylor representation~\cite{BonitoLeiPasciak2019} and a local
discontinuous Galerkin method~\cite{HanWuZhou2025}. For broader discussions
of fractional finite element methods and their implementation, we refer
to~\cite{BonitoBorthagarayNochettoOtarolaSalgado2018,BorthagarayLiNochetto2020}.

Pointwise error analysis for the integral fractional Laplacian has been
developed for finite difference schemes. A finite difference-quadrature
method was introduced in~\cite{HuangOberman2014}. Further finite difference
approximations for the operator and
fractional Poisson problems were developed
in~\cite{DuoVanWykZhang2018,DuoZhang2019}. The quantitative accuracy
estimates in these works assume
global regularity that is not generally available for the Dirichlet
problem because of boundary singularities. A monotone two-scale discretization on
unstructured grids was developed in~\cite{HanWu2022}, with pointwise error
estimates based on weighted H\"older
regularity compatible with these singularities.
Monotone schemes have also been developed for fractional obstacle
problems~\cite{HanWuZhou2023} and a class of integro-differential
equations~\cite{BorthagarayNochettoSalgadoTorres2026}.

For piecewise linear finite elements on one-dimensional uniform meshes,
an $M$-matrix analysis yields discrete maximum principles for $s_0<s<1$,
where $s_0\approx0.2347$~\cite{LiuShengWangYuan2021}.

Local energy estimates provide a key ingredient in this analysis.
Borthagaray, Leykekhman, and
Nochetto~\cite{BorthagarayLeykekhmanNochetto2021} established local error
estimates in the fractional energy seminorm. Estimates in both the
localized energy norm and the local $H^1$-norm were obtained
in~\cite{FaustmannKarkulikMelenk2022}. Our analysis builds
on~\cite{BorthagarayLeykekhmanNochetto2021}, with additional control of
the nonlocal tail terms. To the best of our knowledge, maximum-norm
stability of the fractional Ritz projection, even up to logarithmic
factors, on general quasi-uniform meshes in arbitrary dimensions has
remained an open problem.

Our main result applies to continuous piecewise linear finite elements on shape-regular, quasi-uniform simplicial meshes. We allow an approximate mesh domain $\Omega_h\subset\Omega$, with boundary discrepancy of order $h$, under the geometric assumptions stated in Section~\ref{sec:preliminaries}. Thus the mesh domain need not coincide with $\Omega$. We prove that
\begin{equation}\label{eqn:intro-stability}
 \|R_h^s u\|_{L^\infty(\Omega_h)}
 \leq C \ell_h^\alpha\|u\|_{L^\infty(\Omega)},
 \qquad u\in\widetilde H^s(\Omega)\cap L^\infty(\Omega),
\end{equation}
where
\[
 \ell_h :=1+|\log h|,\qquad
 \alpha :=
 \begin{cases}
  1,&0<s<1/2,\\
  d+2,&s=1/2.
 \end{cases}
\]
In particular, a single logarithmic factor suffices throughout $0<s<1/2$ under the stated Lipschitz and exterior ball assumptions. The estimate holds in arbitrary spatial dimensions and covers the case in which $s=1/2$; see
Theorem~\ref{thm:linf-stab}.

Since $R_h^s$ is a projection onto $V_h$, \eqref{eqn:intro-stability} immediately implies
\begin{equation}\label{eqn:intro-best-approximation}
 \|u-R_h^s u\|_{L^\infty(\Omega_h)}
 \leq C \ell_h^\alpha
 \inf_{v_h\in V_h}\|u-v_h\|_{L^\infty(\Omega)}.
\end{equation}
Any additional regularity of $u$ can be used through this approximation term, rather than imposed in the stability estimate itself.

The case in which $s=1/2$ also has a boundary element interpretation. On a
flat open screen, the Laplace hypersingular boundary integral operator
coincides, up to a fixed normalization, with $(-\Delta)^{1/2}$ acting on
zero-extended screen densities~\cite{GimperleinStocekUrzuaTorres2021}.
Consequently, our result gives maximum-norm almost stability of the
corresponding Galerkin projection under the stated geometric and mesh
assumptions. This complements the classical pointwise analysis of
boundary element methods on smooth closed
surfaces~\cite{RannacherWendland1985,RannacherWendland1988}.

If $\Omega_h\nsubseteq\Omega$, we consider an auxiliary domain $\widetilde\Omega$ with $\Omega\cup\Omega_h\subset\widetilde\Omega\subset \{x:\mathrm{dist}(x,\Omega)\leq C\delta_h\}$.
We assume that the auxiliary domain and its mesh approximation satisfy the geometric assumptions of Section~\ref{sec:preliminaries}, with constants independent of $h$.
Let $\widetilde{f}$ be the zero extension of $f$ to $\widetilde\Omega$, and let $\widetilde u$ solve the corresponding problem on $\widetilde\Omega$. The numerical solution is then $u_h=R_h^s\widetilde u$. Applying \eqref{eqn:intro-best-approximation} on $\widetilde\Omega$ along with the maximum principle gives
\begin{equation}\label{eqn:linf-optimal-out}
  \begin{split}
    \|u-u_h\|_{L^\infty(\Omega_h)} &\leq \|u-\widetilde u\|_{L^\infty(\widetilde\Omega)} + C\ell_h^\alpha\inf_{v_h\in V_h}\|\widetilde u-v_h\|_{L^\infty(\widetilde\Omega)} \\
  &\leq \|\widetilde u\|_{L^\infty(\widetilde\Omega\setminus\Omega)} + C\ell_h^\alpha\inf_{v_h\in V_h}\|\widetilde u-v_h\|_{L^\infty(\widetilde\Omega)}.
  \end{split}
 \end{equation}
A standard exterior-ball barrier argument bounds the first term by $C\delta_h^s\|f\|_{L^\infty(\Omega)}$~\cite[Lemma~2.7]{RosOtonSerra2014}.
In the rest of the paper, we assume $\Omega_h\subset\Omega$.

Our proof builds on the local approach of Schatz and Wahlbin. Nonlocality
introduces unknown errors on distant annuli into the local error estimates,
requiring simultaneous control of local terms and nonlocal tails. Near the
boundary, the analysis must also retain interactions across the domain
boundary. We therefore distinguish full annuli from their intersections
with the domain and measure the fractional energy of zero extensions on
the full annuli. This yields the boundary-weighted control needed in the
localized duality argument. Combining fractional cutoff and oscillation
estimates with Green's function bounds, we establish local regularity
estimates for localized sources with explicit dependence on the relative
annular scales. We also extend the local energy estimate
of~\cite{BorthagarayLeykekhmanNochetto2021} to boundary neighborhoods and
establish local interpolation estimates that account for the discrepancy
between the mesh domain and the original domain.

Through a discrete dual problem, we reduce pointwise control of the Ritz
projection to an $L^1$ bound for the fractional Laplacian of the dual
solution. Local energy estimates and localized duality, together with the
preceding regularity and interpolation estimates, yield coupled bounds
for the sequences of annular energy and $L^2$ errors. The key is to use
the precise mutual control of these two sequences to handle the nonlocal
tails involving all annuli and close the estimates. Summation over the
annuli then gives the required $L^1$ bound and hence maximum-norm stability
up to logarithmic factors.

The restriction $0<s\leq1/2$ reflects the regularity available for the
auxiliary dual problems. For square-integrable data, full $H^{2s}$
regularity holds for $s<1/2$, with an arbitrarily small regularity loss
at $s=1/2$; see~\cite[Corollary~3.8]{BorthagarayLiNochetto2024} and
\cite[Remark~6]{BorthagarayNochetto2023}. These results are recalled in
Theorem~\ref{thm:sobolev-shift-lip}. In this range, a single dyadic
decomposition centered at the dual source suffices to close the analysis.
For $s>1/2$, the typical boundary behavior
$\dist(x,\partial\Omega)^s$ limits global Sobolev regularity to orders
below $s+1/2<2s$~\cite{AcostaBorthagaray2017,BorthagarayNochetto2023}.
Extending the present approach therefore calls for further analysis of
the boundary scale in addition to the distance from the dual source.
Several ingredients, including the local interpolation estimates and
the extension of the local energy estimate to boundary neighborhoods,
already apply to the full range $0<s<1$. They provide a basis for this
extension, which will be addressed in subsequent work. The localization
and sequence estimates may also be useful for more general nonlocal
operators, including variable-coefficient kernels, once corresponding
regularity and Green's function estimates are available.

The remainder of the paper is organized as follows.
Section~\ref{sec:preliminaries} introduces the notation, assumptions,
and main theorem. Section~\ref{sec:regularity} establishes local
regularity estimates for localized sources, and
Section~\ref{sec:interpolation} proves the required interpolation
estimates. Section~\ref{sec:annular-error} derives and closes the coupled
annular error estimates. Section~\ref{sec:linf-stability} completes the
maximum-norm stability proof. Appendices~\ref{app:critical-case}
and~\ref{app:local-energy-error} treat the critical case $(s,d)=(1/2,1)$
and the extension of the local energy estimate to boundary neighborhoods,
respectively.

Throughout, $C$ denotes a generic positive constant that may change at each
occurrence and is independent of $h$ and the functions being estimated,
unless otherwise stated. We write $A\simeq B$ if $C^{-1}B\leq A\leq CB$,
with subscripts indicating the parameters on which the constants may depend.

\section{Preliminaries and main results}\label{sec:preliminaries}
We fix the notation, state the main results, and collect auxiliary estimates for fractional Sobolev norms.
\subsection{Problem setting and notation}
Let $\Omega\subset\RR^d$, $d\geq1$, be a bounded Lipschitz domain satisfying a uniform exterior ball condition: there exists $r_e>0$ such that, for every $x\in\partial\Omega$ and $0<r<r_e$, there is a ball $B_r$ of radius $r$ with $\overline B_r\cap\overline\Omega=\{x\}$. By scaling, we assume $\diam\Omega\leq1$. The Lipschitz property implies the uniform exterior measure-density condition: there exist $r_\Omega,c_\Omega>0$ such that
\begin{equation}\label{eqn:exterior-density}
 |B_r(x)\cap\Omega^c|\geq c_\Omega r^d
 \quad
 \forall x\in\partial\Omega, 0<r\le r_\Omega.
\end{equation}
\begin{definition}
For $0<\sigma<1$ and a measurable set $D\subset\RR^d$, define the normalized Aronszajn--Gagliardo--Slobodeckij seminorm by
\begin{equation}\label{eq:slobo}
 \begin{aligned}
 |v|_{H^\sigma(D)}^2
 :=\frac{c_{d,\sigma}}2\int_D\int_D
 \frac{|v(x)-v(y)|^2}{|x-y|^{d+2\sigma}}\,\mathrm{d}x\mathrm{d}y, \quad c_{d,\sigma}
 :=\frac{2^{2\sigma}\sigma\Gamma(d/2+\sigma)}
 {\pi^{d/2}\Gamma(1-\sigma)}.
 \end{aligned}
\end{equation}
We set $H^\sigma(D):=\{v\in L^2(D):|v|_{H^\sigma(D)}<\infty\}$, with norm $\|v\|_{H^\sigma(D)}^2:=\|v\|_{L^2(D)}^2+|v|_{H^\sigma(D)}^2$.
For an open set $D$, define
\[
 H^{1+\sigma}(D)
 :=\{v\in H^1(D):\partial_jv\in H^\sigma(D),\ 1\le j\le d\},
\]
with seminorm and norm
\begin{equation}\label{eq:slobo-higher-order}
 \begin{aligned}
 |v|_{H^{1+\sigma}(D)}^2
 &:=\sum_{j=1}^d|\partial_jv|_{H^\sigma(D)}^2
 =\frac{c_{d,\sigma}}2\int_D\int_D
 \frac{|\nabla v(x)-\nabla v(y)|^2}
 {|x-y|^{d+2\sigma}}\,\mathrm{d}x\mathrm{d}y,\\
 \|v\|_{H^{1+\sigma}(D)}^2
 &:=\|v\|_{H^1(D)}^2+|v|_{H^{1+\sigma}(D)}^2.
 \end{aligned}
\end{equation}
At integer orders we use $H^0(D):=L^2(D)$ and
\[
 |v|_{H^0(D)}:=\|v\|_{L^2(D)},\quad
 |v|_{H^1(D)}:=\|\nabla v\|_{L^2(D)},\quad
 |v|_{H^2(D)}:=\|D^2v\|_{L^2(D)},
\]
where $D^2v$ is equipped with the Frobenius norm, and
$\|v\|_{H^m(D)}^2:=\sum_{j=0}^m|v|_{H^j(D)}^2$ for $m=1,2$.
\end{definition}

For $0\leq t\leq2$, set
\[
 \widetilde H^t(\Omega)
 :=\{v\in H^t(\RR^d):v=0\ \text{a.e. in }\Omega^c\},
\]
with the inherited $H^t(\RR^d)$ norm.

For sufficiently smooth, compactly supported $u$, the integral fractional
Laplacian is defined by
\[
 (-\Delta)^s u(x)
 :=c_{d,s}\operatorname{p.v.}\int_{\RR^d}
 \frac{u(x)-u(y)}{|x-y|^{d+2s}}\,\mathrm{d}y.
\]
Its associated bilinear form on $\widetilde H^s(\Omega)$ is
\[
 a_s(u,\varphi)
 :=\frac{c_{d,s}}2\int_{\RR^d}\int_{\RR^d}
 \frac{(u(x)-u(y))(\varphi(x)-\varphi(y))}
 {|x-y|^{d+2s}}\,\mathrm{d}x\mathrm{d}y.
\]
Given $f\in L^2(\Omega)$, the weak formulation of \eqref{eqn:intro-model} is to find $u\in\widetilde H^s(\Omega)$ such that
\begin{equation}\label{eqn:fractional-laplacian}
a_s(u,v)
=
\int_\Omega f v\,\mathrm{d}x,
\quad
\forall v\in\widetilde H^s(\Omega).
\end{equation}
Existence and uniqueness follow from the fractional Poincar\'e
inequality (see~\cite[Lemma~6.1]{DiNezzaPalatucciValdinoci2012})
and the Lax--Milgram theorem.

\subsection{Finite element space and main result}
Let $\{\mathcal T_h\}_{h>0}$ be a shape-regular, quasi-uniform
family of simplicial meshes. Set $h_T:=\diam T$ and
$h:=\max_{T\in\mathcal T_h}h_T$. Let $\rho_T$ be the inradius
of $T$ and assume
$\gamma_h:=\max_{T\in\mathcal T_h}h/\rho_T\leq\gamma<\infty$.
For the mesh domain
$\Omega_h:=\operatorname{int}\bigl(\bigcup_{T\in\mathcal T_h}T\bigr)\subset\Omega$, we assume
\begin{equation}\label{eqn:mesh-domain-distance}
 \delta_h:=d_H(\partial\Omega,\partial\Omega_h)
 \leq C_{\mathrm{geo}}h,  \text{ and }  
 \Omega\setminus\Omega_h
\subset\{x\in\RR^d:\dist(x,\partial\Omega)
\leq C_{\mathrm{geo}}\delta_h\},
\end{equation}
where $d_H$ denotes the Hausdorff distance.
The piecewise linear Lagrange space is
\[
 \widehat V_h
 :=\{v_h\in C^0(\overline\Omega_h):v_h|_T\in\mathbb P_1
       \text{ for all }T\in\mc T_h\},
\]
and $V_h:=\widehat V_h\cap H_0^1(\Omega_h)$ is its subspace with homogeneous boundary conditions.
Functions on $\Omega$ or $\Omega_h$ are extended by zero to $\RR^d$ whenever needed, without changing notation.
The Ritz projection \(R_h^s:\widetilde H^s(\Omega)\to V_h\) is defined by
\begin{equation}\label{finite element}
     a_s(R_h^s u,v_h)=a_s(u,v_h)
 \qquad\forall v_h\in V_h.
\end{equation}
Our main result is the following theorem.

\begin{theorem}[almost $L^\infty$-stability]\label{thm:linf-stab}
    Let $\Omega\subset\RR^d$, $d\geq1$, be a bounded Lipschitz domain
    satisfying a uniform exterior ball condition, and let $0<s\leq1/2$.
    Let $\{\mathcal T_h\}_{h>0}$ be a shape-regular, quasi-uniform family
    of simplicial meshes of $\Omega_h\subset\Omega$ satisfying
    \eqref{eqn:mesh-domain-distance}.
    Then the Ritz projection $R_h^s$ onto the continuous piecewise linear
    space $V_h\subset H_0^1(\Omega_h)$, defined by \eqref{finite element},
    satisfies, for every $u\in\widetilde H^s(\Omega)\cap L^\infty(\Omega)$,
    \begin{equation}\label{eqn:linf-stab}
        \|R_h^s u\|_{L^\infty(\Omega_h)} \leq C \ell_h^{\alpha} \|u\|_{L^\infty(\Omega)},
    \end{equation}
    where $C=C(d,s,\Omega,\gamma,C_{\mathrm{geo}})$ is independent of $h$
    and $u$. Here $\ell_h:=1+|\log h|$ and
    $\alpha:=1+\frac{d+2s}{2s}\xi$, with
    \begin{equation}\label{eqn:xi-alpha}
    \xi := \begin{cases}
        0 & s<1/2,\\
        1 & s=1/2,\\
    \end{cases}
    \quad \text{i.e.} \quad \alpha = \begin{cases}
        1 & s<1/2,\\
        d+2 & s=1/2.\\
    \end{cases}
    \end{equation}
\end{theorem}


\subsection{Auxiliary estimates for fractional Sobolev norms}
The following standard estimates control cutoff multiplication,
localization of fractional seminorms, and boundary weights arising from
exterior interactions. For completeness, we give the cutoff proof in detail
and sketch the other two proofs. For cutoff and localization arguments, see
\cite[Lemmas~5.1 and~5.3]{DiNezzaPalatucciValdinoci2012} and
\cite[Section~3.1]{BorthagarayLeykekhmanNochetto2021}.
\begin{lemma}[cutoff estimate]\label{lem:frac-cutoff}
    Let $\eta\in C^2(\RR^d)$ satisfy $0\leq\eta\leq1$ and $|\nabla^i\eta|\leq Cr^{-i}$ for some $r>0$ and $i=1,2$. Let $D$ be an open set with $\diam D\leq Cr$ and $\dist(\supp\eta,D^c)\geq r/C$. Extending $\eta v$ by zero outside $D$, we have, for $0\leq s\leq1$ and $v\in H^s(D)$,
    \begin{equation}\label{eqn:cutoff-lower-order}
        |\eta v|_{H^s(\RR^d)}\leq \sqrt{2}|v|_{H^s(D)} + C_{d}\, r^{-s}\|v\|_{L^2(D)},
    \end{equation}
    and for every $1\leq t\leq2$ and $v\in H^t(D)$,
    \begin{equation}\label{eqn:cutoff-higher-order}
        |\eta v|_{H^t(\RR^d)}
        \leq C_d\bigl(|v|_{H^t(D)}
        +r^{1-t}\|\nabla v\|_{L^2(D)}
        +r^{-t}\|v\|_{L^2(D)}\bigr),
    \end{equation}
    where $C_d$ depends only on $d$ and $C$, uniformly in $s$ and $t$.
\end{lemma}
\begin{proof}
For $0<s<1$, the zero extension of $\eta v$ satisfies
\[
 |\eta v|_{H^s(\RR^d)}^2
 =|\eta v|_{H^s(D)}^2
 +c_{d,s}\int_D|\eta(x)v(x)|^2
       \int_{D^c}\frac{\mathrm{d}y}{|x-y|^{d+2s}}\,\mathrm{d}x.
\]
For the first term, write
$\eta(x)v(x)-\eta(y)v(y)
=\eta(x)(v(x)-v(y))+v(y)(\eta(x)-\eta(y))$.
Using $|\eta|\leq1$, $|\eta(x)-\eta(y)|\leq Cr^{-1}|x-y|$,
and $\diam D\leq Cr$, we obtain
\[
 \begin{aligned}
 |\eta v|_{H^s(D)}^2
 &\leq 2|v|_{H^s(D)}^2
   +Cc_{d,s}r^{-2}\int_D|v(y)|^2
       \int_{|z|\leq Cr}|z|^{2-d-2s}\,\mathrm{d}z\mathrm{d}y\\
 &\leq 2|v|_{H^s(D)}^2
   +\frac{C_dc_{d,s}}{1-s}r^{-2s}\|v\|_{L^2(D)}^2.
 \end{aligned}
\]
For the exterior term, the separation of $\supp\eta$ from $D^c$ gives,
for $x\in\supp\eta$,
\[
 \int_{D^c}\frac{\mathrm{d}y}{|x-y|^{d+2s}}
 \leq\int_{|z|\geq r/C}\frac{\mathrm{d}z}{|z|^{d+2s}}
 \leq\frac{C_d}{s}r^{-2s}.
\]
Combining these bounds and using $c_{d,s}\simeq_d s(1-s)$ yields
\[
 |\eta v|_{H^s(\RR^d)}^2
 \leq 2|v|_{H^s(D)}^2+C_dr^{-2s}\|v\|_{L^2(D)}^2.
\]
Taking square roots proves \eqref{eqn:cutoff-lower-order} uniformly in $s$.
The cases $s=0,1$ and $t=1,2$ follow from $|\eta|\leq1$ and the product rule.

For $1<t<2$, set $\sigma:=t-1$. For $f\in H^1(\RR^d)$, the translation
bound $\|f(\cdot+z)-f\|_{L^2}\leq
\min\{|z|\|\nabla f\|_{L^2},2\|f\|_{L^2}\}$ and a split at $|z|=r$ give
\[
 \begin{aligned}
 |f|_{H^\sigma(\RR^d)}^2
 &\leq C_dc_{d,\sigma}\left(
      \frac{r^{2-2\sigma}}{1-\sigma}\|\nabla f\|_{L^2(\RR^d)}^2
      +\frac{r^{-2\sigma}}{\sigma}\|f\|_{L^2(\RR^d)}^2\right)\\
 &\leq C_d\left(r^{2-2\sigma}\|\nabla f\|_{L^2(\RR^d)}^2
      +r^{-2\sigma}\|f\|_{L^2(\RR^d)}^2\right),
 \end{aligned}
\]
uniformly in $\sigma\in(0,1)$. For $1\leq j\leq d$, let $f_j$ be the
zero extension of $v\partial_j\eta$. The support separation ensures
$f_j\in H^1(\RR^d)$, and the derivative bounds on $\eta$ yield
\[
 \begin{aligned}
 |f_j|_{H^\sigma(\RR^d)}
 &\leq C_d\bigl(r^{1-\sigma}\|\nabla f_j\|_{L^2(\RR^d)}
              +r^{-\sigma}\|f_j\|_{L^2(\RR^d)}\bigr)\\
 &\leq C_d\bigl(r^{-\sigma}\|\nabla v\|_{L^2(D)}
              +r^{-1-\sigma}\|v\|_{L^2(D)}\bigr).
 \end{aligned}
\]
On the other hand, \eqref{eqn:cutoff-lower-order} applied to $\partial_jv$
gives
\[
 |\eta\partial_jv|_{H^\sigma(\RR^d)}
 \leq\sqrt2|\partial_jv|_{H^\sigma(D)}
       +C_dr^{-\sigma}\|\partial_jv\|_{L^2(D)}.
\]
Using $\partial_j(\eta v)=\eta\partial_jv+f_j$ and summing over $j$
proves \eqref{eqn:cutoff-higher-order}.
\end{proof}

\begin{lemma}[localization]\label{lem:frac-localize}
    Let $D\subset\RR^d$ be an open set with a finite covering $D\subset\bigcup_{i=1}^N B_r(x_i)$, where $r>0$. Then, for $0<s<1$ and $v\in H^s(D)$,
    \begin{equation}
        |v|^2_{H^s(D)}\leq \sum_{i=1}^N |v|^2_{H^s(D\cap B_{2r}(x_i))} + C_d\,r^{-2s}\|v\|^2_{L^2(D)},
    \end{equation}
    where $C_d$ depends only on $d$.
\end{lemma}
\begin{proof}[Sketch of proof]
Split the defining double integral at $|x-y|=r$. If $x\in B_r(x_i)$
and $|x-y|\leq r$, then $x,y\in B_{2r}(x_i)$, so the near-field
contribution is bounded by the sum on the right-hand side.
By symmetry and radial integration, the far-field contribution is at most
$C_dc_{d,s}s^{-1}r^{-2s}\|v\|_{L^2(D)}^2
\leq C_dr^{-2s}\|v\|_{L^2(D)}^2$, since $c_{d,s}\simeq_d s(1-s)$.
Adding the two bounds proves the assertion.
\end{proof}

\begin{lemma}[exterior kernel bound]\label{lem:delta-upper-bound}
Let $\Omega\subset\RR^d$ be a bounded domain satisfying the exterior measure-density condition \eqref{eqn:exterior-density}, and set $\delta(x):=\dist(x,\partial\Omega)$ for $x\in\Omega$. For every $s>0$, there exists $C=C(d,s,c_\Omega,r_\Omega,\diam\Omega)>0$ such that
\begin{equation}\label{eqn:delta-upper-bound}
 C^{-1}\delta(x)^{-2s}\leq\int_{\Omega^c}
\frac{\mathrm{d}y}{|x-y|^{d+2s}}
\leq C\delta(x)^{-2s},
\qquad x\in\Omega.
\end{equation}
\end{lemma}
\begin{proof}[Sketch of proof]
Fix $x\in\Omega$, set $\delta:=\delta(x)$, and choose
$z_x\in\partial\Omega$ with $|x-z_x|=\delta$.
Let $r_x:=\min\{\delta,r_\Omega\}$. For $y\in B_{r_x}(z_x)$,
we have $|x-y|\leq2\delta$, so \eqref{eqn:exterior-density} gives
\[
 \int_{\Omega^c}\frac{\mathrm{d}y}{|x-y|^{d+2s}}
 \geq(2\delta)^{-d-2s}|B_{r_x}(z_x)\cap\Omega^c|
 \geq c_\Omega2^{-d-2s}
       \left(\frac{r_x}{\delta}\right)^d\delta^{-2s}.
\]
Since $r_x/\delta\geq\min\{1,r_\Omega/\diam\Omega\}>0$,
the lower bound follows. The upper bound follows by radial integration
over $B_\delta(x)^c\supset\Omega^c$.
\end{proof}

\section{Local regularity for localized sources}\label{sec:regularity}
We establish local $L^2$ and $H^{2s-\eps}$ estimates for solutions with sources
supported in an annulus or a mesh element. These estimates will be used in the
interpolation and duality arguments below.
The parameter $0\leq\eps<s$ accounts for the loss of regularity;
we may take $\eps=0$ for $s<1/2$, whereas $\eps>0$ is required for $s=1/2$.

\subsection{Dyadic decomposition and local estimates}
For \(x_0\in\overline\Omega\), set \(\delta(x_0):=\dist(x_0,\partial\Omega)\).
For \(i\in\mathbb{Z}\), define
\begin{equation}\label{eqn:annuli-0}
 r_i:=2^{-i},
 \quad
 A_i=A_i(x_0):=B_{r_i}(x_0)\setminus B_{r_{i+1}}(x_0),
 \quad \Omega_i=\Omega_i(x_0):=\Omega\cap A_i.
\end{equation}
For a terminal index \(J\in\NN_+\), set
\[
    A_\ast=A_\ast(x_0):=B_{r_J}(x_0),\quad \Omega_\ast=\Omega_\ast(x_0):=\Omega\cap A_\ast.
\]
Since $\diam \Omega\leq1$, we have $\Omega\subset B_1(x_0)$ and the disjoint decompositions
\[
    B_1(x_0)=A_\ast\cup\bigcup_{i=0}^{J-1}A_i,
    \qquad \Omega=\Omega_\ast\cup\bigcup_{i=0}^{J-1}\Omega_i.
\]
For $j\in\mathbb{Z}$ and $p\in\NN_+$, define the enlarged regions by
\[
    A_j^{(p)}:=\bigcup_{i=j-p}^{j+p}A_i,
    \qquad \Omega_j^{(p)}:=\Omega\cap A_j^{(p)},
\]
\[
    A_\ast^{(p)}:=A_\ast\cup\bigcup_{i=J-p}^{J-1}A_i=B_{r_{J-p}}(x_0),
    \qquad \Omega_\ast^{(p)}:=\Omega\cap A_\ast^{(p)}.
\]
The indices in these unions are not truncated: $\Omega_i=\varnothing$ for $i<0$, while $A_i\subset A_\ast$ for $i\geq J$. We write $A_j':=A_j^{(1)}$ and $A_j'':=A_j^{(2)}$, and use the same convention for $\Omega_j$, $A_\ast$ and $\Omega_\ast$.

To estimate local $L^2$ errors by duality in Section~\ref{sec:annular-error},
we first consider sources supported in an annulus. For $x_0\in\overline\Omega$,
$J\in\NN_+$, and $0\leq j<J$, let $v_j\in L^2(\Omega)$ satisfy
$\supp v_j\subset\overline{\Omega_j(x_0)}$, and let
$w_j\in\widetilde H^s(\Omega)$ solve
\begin{equation}\label{eqn:local-source-annuli}
 \begin{cases}
  (-\Delta)^s w_j=v_j &\text{in }\Omega,\\
  w_j=0 &\text{in }\Omega^c.
 \end{cases}
\end{equation}

\begin{lemma}[annular sources: $L^2$ estimates]\label{lem:local-L2}
Let $0<s\leq1/2$ and let $w_j$ solve \eqref{eqn:local-source-annuli}.
For $0\leq k<J$, the following estimates hold. If $2s<d$, then
\begin{align}
 \|w_j\|_{L^2(\Omega_k)}
 &\leq C(r_jr_k)^{d/2}r_{j\wedge k}^{2s-d}\|v_j\|_{L^2(\Omega_j)},
 \label{eq:L2-comp}\\
 \|w_j\|_{L^2(\Omega_\ast)}
 &\leq C r_J^{d/2}r_j^{2s-d/2}\|v_j\|_{L^2(\Omega_j)}.
 \label{eq:L2-in-core}
\end{align}
If $(s,d)=(1/2,1)$, then
\begin{align}
 \|w_j\|_{L^2(\Omega_k)}
 &\leq C\sqrt{r_k r_j}\bigl(1+|\log r_{j\wedge k}|\bigr)
 \|v_j\|_{L^2(\Omega_j)},\label{eq:L2-comp-critical}\\
 \|w_j\|_{L^2(\Omega_\ast)}
 &\leq C\sqrt{r_J r_j}\bigl(1+|\log r_j|\bigr)
 \|v_j\|_{L^2(\Omega_j)}.\label{eq:L2-in-core-critical}
\end{align}
Here $C=C(d,s,\Omega)$ is independent of $x_0$, $j$, $k$, $J$, and $v_j$.
\end{lemma}

\begin{lemma}[annular sources: Sobolev estimates]\label{lem:local-regularity}
Let $0<s\leq1/2$, let $w_j$ solve \eqref{eqn:local-source-annuli}, and let
$0\leq\eps<s$, with $\eps>0$ if $s=1/2$. For $0\leq k\leq J$, define
\begin{equation}\label{eqn:kappa}
 \kappa_{jk}:=
 \begin{cases}
  (r_j/r_k)^{d/2},&k\leq j-4,\\[3pt]
  1,&|k-j|\leq3,\\[3pt]
  (r_k/r_j)^{d/2-s},&k\geq j+4,
 \end{cases}
\end{equation}
where $j$ is the source index and $k$ the observation index. Then, for $0\leq k<J$,
\begin{align}
 |w_j|_{H^{2s-\eps}(A_k)}
 &\leq C\eps^{-\xi}\kappa_{jk}\|v_j\|_{L^2(\Omega_j)},
 \label{eqn:local-regularity-annuli}\\
 |w_j|_{H^{2s-\eps}(A_\ast)}
 &\leq C\eps^{-\xi}\kappa_{jJ}\|v_j\|_{L^2(\Omega_j)}.
 \label{eqn:local-regularity-annuli-core}
\end{align}
Here $\xi$ is defined in \eqref{eqn:xi-alpha}, with $\eps^{-\xi}=1$ for $s<1/2$.
The constant $C=C(d,s,\Omega)$ is independent of $x_0$, $j$, $k$, $J$, $\eps$,
and $v_j$.
\end{lemma}

For the almost $L^\infty$ stability proof in Section~\ref{sec:linf-stability}, we use a
dual problem with a source supported in a single mesh element. For
$x_0\in T_0\in\mc T_h$ and $\varphi\in L^2(\Omega)$ with
$\supp\varphi\subset T_0$, let $z\in\widetilde H^s(\Omega)$ solve
\begin{equation}\label{eqn:local-source}
 \begin{cases}
  (-\Delta)^s z=\varphi &\text{in }\Omega,\\
  z=0 &\text{in }\Omega^c.
 \end{cases}
\end{equation}

\begin{lemma}[element sources: $L^2$ estimates]\label{lem:local-L2-core}
Let $0<s\leq1/2$ and let $z$ solve \eqref{eqn:local-source}.
Assume $r_J\geq2h$. For $0\leq k<J$, the following estimates hold.
If $2s<d$, then
\begin{equation}\label{eq:L2-out-comp-core}
 \|z\|_{L^2(\Omega_k)}
 \leq C r_k^{2s-d/2}h^{d/2}\|\varphi\|_{L^2(T_0)},
\end{equation}
\begin{equation}\label{eq:L2-core-core}
 \|z\|_{L^2(\Omega_\ast)}
 \leq C\|\varphi\|_{L^2(T_0)}
 \begin{cases}
  h^{2s},&d>4s,\\[3pt]
  h^{2s}\bigl(1+\log_2(r_J/h)\bigr)^{1/2},&d=4s,\\[3pt]
  h^{d/2}r_J^{2s-d/2},&2s<d<4s.
 \end{cases}
\end{equation}
If $(s,d)=(1/2,1)$, then
\begin{align}
 \|z\|_{L^2(\Omega_k)}
 &\leq C\sqrt{h r_k}\bigl(1+|\log r_k|\bigr)
 \|\varphi\|_{L^2(T_0)},\label{eq:L2-out-comp-core-critical}\\
 \|z\|_{L^2(\Omega_\ast)}
 &\leq C\sqrt{h r_J}\bigl(1+|\log r_J|\bigr)
 \|\varphi\|_{L^2(T_0)}.\label{eq:L2-core-core-critical}
\end{align}
Here $C=C(d,s,\Omega)$ is independent of $h$, $x_0$, $T_0$, $k$, $J$, and $\varphi$.
\end{lemma}

\begin{lemma}[element sources: Sobolev estimates]\label{lem:local-regularity-core}
Let $0<s\leq1/2$, let $z$ solve \eqref{eqn:local-source}, and let
$0\leq\eps<s$, with $\eps>0$ if $s=1/2$.
Assume $r_J\geq2h$. For $0\leq k<J$,
\begin{align}
 |z|_{H^{2s-\eps}(A_k)}
 &\leq C\eps^{-\xi}(h/r_k)^{d/2}\|\varphi\|_{L^2(T_0)},
 \label{eqn:local-regularity-core}\\
 |z|_{H^{2s-\eps}(A_\ast)}
 &\leq C\eps^{-\xi}\|\varphi\|_{L^2(T_0)}.
 \label{eqn:local-regularity-core-core}
\end{align}
Here $\xi$ is defined in \eqref{eqn:xi-alpha}, with $\eps^{-\xi}=1$ for $s<1/2$.
The constant $C=C(d,s,\Omega)$ is independent of $h$, $x_0$, $T_0$, $k$, $J$,
$\eps$, and $\varphi$.
\end{lemma}

The estimates in Lemmas~\ref{lem:local-L2}--\ref{lem:local-regularity-core}
also hold with $A_k$, $A_\ast$, $\Omega_k$, and $\Omega_\ast$ replaced by
$A_k^{(p)}$, $A_\ast^{(p)}$, $\Omega_k^{(p)}$, and $\Omega_\ast^{(p)}$,
respectively, for any fixed $p\in\NN_+$, with constants also depending on $p$.

\subsection{Green's function and regularity estimates}
We collect the tools for Lemmas~\ref{lem:local-L2}--\ref{lem:local-regularity-core}.
Green's function bounds yield the $L^2$ estimates. For the Sobolev estimates,
we recall global regularity and establish localized regularity and an interior
gradient estimate.

\begin{theorem}[Green's function estimate, {\cite[pp.~1309--1310]{ChenKimSong2010}}]\label{thm:Green-estimate}
Let $\Omega$ be a bounded Lipschitz domain satisfying a uniform exterior ball
condition, and let $0<s\leq1/2$. The Green's function $G_\Omega$ of the integral
fractional Laplacian with zero exterior values satisfies, for $x,y\in\Omega$
with $x\ne y$,
\begin{equation}\label{eqn:Green-estimate}
 G_\Omega(x,y)\leq
 \begin{cases}
 C|x-y|^{2s-d}\min\left\{1,\dfrac{\delta(x)^s}{|x-y|^s}\right\}
 \min\left\{1,\dfrac{\delta(y)^s}{|x-y|^s}\right\},&d>2s,\\[3pt]
 C\log\left(1+\dfrac{\sqrt{\delta(x)\delta(y)}}{|x-y|}\right),&d=2s,
 \end{cases}
\end{equation}
where $C$ depends only on $d$, $s$, and $\Omega$.
\end{theorem}

The next estimates follow from \cite[Corollary~3.8]{BorthagarayLiNochetto2024}
for $s<1/2$ and \cite[Remark~6]{BorthagarayNochetto2023} for $s=1/2$.
\begin{theorem}[global Sobolev regularity]\label{thm:sobolev-shift-lip}
Let $\Omega$ be a bounded Lipschitz domain, $0<s\leq1/2$, and let
$u\in\widetilde H^s(\Omega)$ solve \eqref{eqn:fractional-laplacian} with
$f\in L^2(\Omega)$. For every $0\leq\eps<s$, with $\eps>0$ if $s=1/2$,
we have $u\in\widetilde H^{2s-\eps}(\Omega)$ and
\begin{equation}\label{eqn:Sobolev-shift-all}
 |u|_{H^{2s-\eps}(\RR^d)}\leq C\eps^{-\xi}\|f\|_{L^2(\Omega)},
\end{equation}
where $\xi$ is defined in \eqref{eqn:xi-alpha} and $C=C(d,s,\Omega)$ is
independent of $\eps$.
\end{theorem}

\begin{lemma}[localized Sobolev regularity]\label{lem:local-sobolev}
Let $\Omega$ be a bounded Lipschitz domain, $0<s\leq1/2$, and let
$u\in\widetilde H^s(\Omega)$ solve \eqref{eqn:fractional-laplacian} with
$f\in L^2(\Omega)$. For $0\leq k<J$ and $\eps$ as in Theorem~\ref{thm:sobolev-shift-lip},
\begin{equation}\label{eqn:local-sobolev-shift}
 \begin{aligned}
 |u|_{H^{2s-\eps}(A_k)}
 &\leq C\eps^{-\xi}\biggl(\|f\|_{L^2(\Omega_k')}
       +r_k^{-2s}\|u\|_{L^2(\Omega_k'')}
       +\int_\Omega
       \frac{r_k^{d/2}|u(y)|}{r_k^{d+2s}+|y-x_0|^{d+2s}}\,\mathrm{d}y\biggr),\\
 |u|_{H^{2s-\eps}(A_\ast)}
 &\leq C\eps^{-\xi}\biggl(\|f\|_{L^2(\Omega_\ast')}
       +r_J^{-2s}\|u\|_{L^2(\Omega_\ast'')}
       +\int_\Omega
       \frac{r_J^{d/2}|u(y)|}{r_J^{d+2s}+|y-x_0|^{d+2s}}\,\mathrm{d}y\biggr).
 \end{aligned}
\end{equation}
Here $C=C(d,s,\Omega)$ is independent of $x_0$, $k$, $J$, and $\eps$.
\end{lemma}
\begin{proof}
Choose a smooth cutoff $0\leq\eta_k\leq1$ with
\[
 \eta_k=1\quad\text{on }A_k,\qquad
 \supp\eta_k\subset A_k',\qquad
 \|\nabla^m\eta_k\|_{L^\infty}\leq Cr_k^{-m},\quad m=1,2.
\]
Since $u=0$ in $\Omega^c$, we have $\eta_k u\in\widetilde H^s(\Omega)$ and
$\supp(\eta_k u)\subset\overline{\Omega_k'}$.
In the distributional sense on $\Omega$, the product identity reads
\[
 (-\Delta)^s(\eta_k u)=\eta_k f+c_{d,s}\mc C(u),
 \quad
 \mc C(u)(x):=\operatorname{p.v.}\int_{\RR^d}
 \frac{(\eta_k(x)-\eta_k(y))u(y)}{|x-y|^{d+2s}}\,\mathrm{d}y.
\]
This identity holds first for smooth $u$ and extends by approximation using
the $L^2$ bounds below. Set $u_1:=u\chi_{A_k''}$ and $u_2:=u-u_1$.

For the far field $u_2$, the identity $\eta_k u_2=0$ implies that
$\mc C(u_2)$ vanishes outside $A_k'$. For $x\in A_k'$ and $y\notin A_k''$,
we have $|x-y|\geq (r_k+|y-x_0|)/C$, so
\[
 |\mc C(u_2)(x)|\leq C|\eta_k(x)|\int_\Omega
 \frac{|u(y)|}{r_k^{d+2s}+|y-x_0|^{d+2s}}\,\mathrm{d}y.
\]
Since $|A_k'|\leq Cr_k^d$, this gives
\[
 \|\mc C(u_2)\|_{L^2(\Omega)}\leq Cr_k^{d/2}\int_\Omega
 \frac{|u(y)|}{r_k^{d+2s}+|y-x_0|^{d+2s}}\,\mathrm{d}y.
\]

For the near field $u_1$, if $s<1/2$, then
$ |\mc C(u_1)|\leq C\left(
 \min\{1,|\cdot|/r_k\}|\cdot|^{-d-2s}\right)\ast|u_1|$.
The convolution kernel has $L^1(\RR^d)$ norm bounded by $C r_k^{-2s}$.
Young's inequality therefore yields
\[
 \|\mc C(u_1)\|_{L^2(\RR^d)}
 \leq C r_k^{-2s}\|u_1\|_{L^2(\RR^d)}
 =C r_k^{-2s}\|u\|_{L^2(\Omega_k'')}.
\]

If $s=1/2$, Taylor's formula gives
\[
 \eta_k(x)-\eta_k(y)=\nabla\eta_k(x)\cdot(x-y)+R_k(x,y),
 \quad |R_k(x,y)|\leq Cr_k^{-2}|x-y|^2.
\]
Hence $\mc C(u_1)(x) =\sum_{\ell=1}^d\partial_\ell\eta_k(x)T_\ell u_1(x)+Su_1(x)$, where
\[
 T_\ell u_1(x):=\operatorname{p.v.}\int_{\RR^d}
   \frac{x_\ell-y_\ell}{|x-y|^{d+1}}u_1(y)\,\mathrm{d}y,
 \quad 
 Su_1(x):=\int_{\RR^d}\frac{R_k(x,y)}{|x-y|^{d+1}}u_1(y)\,\mathrm{d}y.
\]
Each $T_\ell$ is a constant multiple of a Riesz transform. Its $L^2$-boundedness
gives
\begin{equation}\label{eqn:riesz-bound}
 \left\|\sum_{\ell=1}^d\partial_\ell\eta_k\,T_\ell u_1\right\|_{L^2(\RR^d)}
 \leq Cr_k^{-1}\|u\|_{L^2(\Omega_k'')}.
\end{equation}
For $y\in A_k''$, we also have $|R_k(x,y)|\leq C$: if $x\in A_k''$, then
$|x-y|\leq Cr_k$; otherwise $\eta_k(x)=0$ and $\nabla\eta_k(x)=0$.
Together with the Taylor bound, this gives
\[
 |R_k(x,y)|\leq C\min\{1,|x-y|^2/r_k^2\},\qquad y\in A_k''.
\]
Young's inequality now yields
\[
 \|Su_1\|_{L^2(\RR^d)}
 \leq C\left\|\min\{1,|\cdot|^2/r_k^2\}|\cdot|^{-d-1}\right\|_{L^1(\RR^d)}
       \|u_1\|_{L^2(\RR^d)}
 \leq Cr_k^{-1}\|u\|_{L^2(\Omega_k'')}.
\]
Together with \eqref{eqn:riesz-bound}, this proves the same near-field bound
for $s=1/2$.

The preceding estimates show that $(-\Delta)^s(\eta_k u)\in L^2(\Omega)$.
Applying Theorem~\ref{thm:sobolev-shift-lip} (global Sobolev regularity) on the
fixed domain $\Omega$ gives
\begin{equation}\label{eqn:local-u}
 |u|_{H^{2s-\eps}(A_k)}
 \leq|\eta_k u|_{H^{2s-\eps}(\RR^d)}
 \leq C\eps^{-\xi}\|(-\Delta)^s(\eta_k u)\|_{L^2(\Omega)}.
\end{equation}
Combining this with the commutator bounds proves the annular estimate in
\eqref{eqn:local-sobolev-shift}. The same argument with $A_k$, $A_k'$, $A_k''$,
and $r_k$ replaced by $A_\ast$, $A_\ast'$, $A_\ast''$, and $r_J$, respectively,
proves the core estimate and completes the proof.
\end{proof}

The following gradient estimate is a consequence of \cite[Lemmas~2.1 and~3.2]{BogdanKulczyckiNowak2002}; we include a direct proof for completeness.
\begin{lemma}[interior gradient estimate]\label{lem:gradient-estimate}
Let $B_{R_0}(x_0)\subset\Omega$ and let $w,W\in\widetilde H^s(\Omega)$ be
$s$-harmonic in $B_{R_0}(x_0)$ in the weak sense. If $W\geq0$ and
$|w|\leq W$ a.e. in $\RR^d$, then
\begin{equation}\label{eq:gradient-estimate}
 \|\nabla w\|_{L^\infty(B_{R_0/2}(x_0))}\leq CR_0^{-1}W(x_0),
\end{equation}
where $C=C(d,s)$ is independent of $R_0$ and $x_0$.
\end{lemma}
\begin{proof}
Set $R:=3R_0/4$. Since $\overline{B_R(x_0)}\subset B_{R_0}(x_0)$, the Poisson
representation applies to $w$ and $W$ on $B_R(x_0)$. The Poisson kernel is
\[
 P_R(x,y):=p_{d,s}\frac{(R^2-|x-x_0|^2)^s}
 {(|y-x_0|^2-R^2)^s|x-y|^d},
 \quad x\in B_R(x_0),\quad y\notin\overline{B_R(x_0)},
\]
where $p_{d,s}:=\Gamma(d/2)\sin(\pi s)/\pi^{d/2+1}$; see \cite{Bucur2016}.
For $x\in B_{R_0/2}(x_0)$ and $y\notin\overline{B_R(x_0)}$, differentiation
and comparison with the kernel at $x_0$ give
\[
 |\nabla_xP_R(x,y)|\leq CR_0^{-1}P_R(x,y)
 \leq CR_0^{-1}P_R(x_0,y).
\]
Differentiating the Poisson representation and using $|w|\leq W$, we obtain
\[
 \begin{aligned}
 |\nabla w(x)|
 &\leq\int_{B_R(x_0)^c}|\nabla_xP_R(x,y)|\,|w(y)|\,\mathrm{d}y\\
 &\leq CR_0^{-1}\int_{B_R(x_0)^c}P_R(x_0,y)W(y)\,\mathrm{d}y
 =CR_0^{-1}W(x_0).
 \end{aligned}
\]
The last equality follows from the Poisson formula for $W$, proving the assertion.
\end{proof}

\subsection{Proofs of the local estimates}
We prove Lemmas~\ref{lem:local-L2}--\ref{lem:local-regularity-core} for $2s<d$,
following the order of their statements. The critical case $(s,d)=(1/2,1)$
is treated in Appendix~\ref{app:critical-case}.

\begin{proof}[Proof of Lemma~\ref{lem:local-L2} when $2s<d$]
For $|k-j|\geq2$, we have $|x-y|\simeq r_{j\wedge k}$ on
$\Omega_k\times\Omega_j$. Theorem~\ref{thm:Green-estimate}
(Green's function estimate) therefore gives
\[
 \|w_j\|_{L^2(\Omega_k)}
 \leq C r_k^{d/2}r_{j\wedge k}^{2s-d}\|v_j\|_{L^1(\Omega_j)} \leq C r_k^{d/2}r_{j\wedge k}^{2s-d}r_j^{d/2}\|v_j\|_{L^2(\Omega_j)}.
\]
For $|k-j|\leq1$, we have $|x-y|\leq3r_j$ on
$\Omega_k\times\Omega_j$. Then
Young's inequality yields
\[
 \|w_j\|_{L^2(\Omega_k)}
 \leq C\bigl\||\cdot|^{2s-d}\chi_{B_{3r_j}(0)}\bigr\|_{L^1(\RR^d)}
       \|v_j\|_{L^2(\Omega_j)} \leq C r_j^{2s}\|v_j\|_{L^2(\Omega_j)}.
\]
Since $r_k\simeq r_j$ for $|k-j|\leq1$, these bounds prove
\eqref{eq:L2-comp}.
Summing the squared bounds over $k\geq J>j$ and using
$\sum_{k\geq J}r_k^d\simeq r_J^d$ gives \eqref{eq:L2-in-core}.
\end{proof}

    \begin{proof}[Proof of Lemma~\ref{lem:local-regularity} when $2s<d$]
    We combine Lemma~\ref{lem:local-L2} (annular sources: $L^2$ estimates)
    with Lemma~\ref{lem:local-sobolev} (localized Sobolev regularity).
    In the inward regime, we also use boundary decay of Green's function in Theorem~\ref{thm:Green-estimate}
    or the interior gradient estimate in Lemma~\ref{lem:gradient-estimate}.

    \textbf{Case 1: outward regime.}
        For $k\leq j-2$, we have $\Omega_k'\cap\Omega_j=\varnothing$,
        so the source term in \eqref{eqn:local-sobolev-shift} vanishes.
        Decomposing the tail integral into annuli and applying Cauchy--Schwarz
        together with \eqref{eq:L2-comp}, we obtain
        \begin{equation}
            \begin{aligned}
                |w_j|_{H^{2s-\eps}(A_k)}
                &\leq C\eps^{-\xi}\biggl(r_k^{-2s}\|w_j\|_{L^2(A_k'')}
                +r_k^{d/2}\int_{\RR^d}\frac{|w_j(y)|}{r_k^{d+2s}+|y-x_0|^{d+2s}}\,\mathrm{d}y\biggr)\\
                &\leq C\eps^{-\xi}r_k^{d/2}\sum_{i=0}^{\infty}
                r_i^{d/2}\|w_j\|_{L^2(A_i)}r_{k\wedge i}^{-d-2s}\\
                &\leq C\eps^{-\xi}(r_jr_k)^{d/2} \|v_j\|_{L^2(\Omega_j)} \sum_{i=0}^{\infty}
                r_i^d r_{j\wedge i}^{2s-d}r_{k\wedge i}^{-d-2s}\\
                &\leq C\eps^{-\xi}\left(\frac{r_j}{r_k}\right)^{d/2}
                 \|v_j\|_{L^2(\Omega_j)} \sum_{i=0}^{\infty}2^{-2s|i-k|} \leq C\eps^{-\xi}\kappa_{jk}\|v_j\|_{L^2(\Omega_j)}.
            \end{aligned}
        \end{equation}
        Here, the local $L^2$ term is included in the sum since $r_i\simeq r_k$
        for $|i-k|\leq2$.
        Since $2s<d$ and $r_{j\wedge i}\geq r_i$, we also used
        \[
            \sum_{i=0}^{\infty}r_i^d r_{j\wedge i}^{2s-d}r_{k\wedge i}^{-d-2s}
            \leq r_k^{-d}\sum_{i=0}^{\infty}2^{-2s|i-k|}
            \leq C r_k^{-d}.
        \]

    \textbf{Case 2: intermediate regime.}
        For $j-1\leq k\leq j+3$, we have $\kappa_{jk}=1$.
        Theorem~\ref{thm:sobolev-shift-lip} (global Sobolev regularity) gives
        \[
            |w_j|_{H^{2s-\eps}(A_k)}
            \leq |w_j|_{H^{2s-\eps}(\RR^d)}
            \leq C\eps^{-\xi}\kappa_{jk}\|v_j\|_{L^2(\Omega_j)}.
        \]

    \textbf{Case 3: inward regime.}
    For $k\geq j+4$, we distinguish two cases according to $\delta(x_0)$.
    
    \fbox{Case 3-1: $\delta(x_0)\leq 8r_k$.}
    For $i\geq j+2$, we have $\delta(x)\leq C r_{k\wedge i}$ and
    $|x-y|\simeq r_j$ for $(x,y)\in\Omega_i\times\Omega_j$.
    Using Theorem~\ref{thm:Green-estimate} (Green's function estimate) for $i\geq j+2$ and
    \eqref{eq:L2-comp} for $i\leq j+1$, we obtain
    \[
        \|w_j\|_{L^2(A_i)}
        \leq C(r_jr_i)^{d/2}r_{j\wedge i}^{s-d}r_{k\wedge i}^{s}
        \|v_j\|_{L^2(\Omega_j)},\qquad i\geq0.
    \]
    Proceeding as in Case~1 gives
    \begin{equation}
        \begin{aligned}
            |w_j|_{H^{2s-\eps}(A_k)}
            &\leq C\eps^{-\xi}r_k^{d/2}\sum_{i=0}^{\infty}
            r_i^{d/2}\|w_j\|_{L^2(A_i)}r_{k\wedge i}^{-d-2s}\\
            &\leq C\eps^{-\xi}(r_jr_k)^{d/2}\|v_j\|_{L^2(\Omega_j)}
            \sum_{i=0}^{\infty}r_i^d r_{j\wedge i}^{s-d}r_{k\wedge i}^{-d-s}\\
            &\leq C\eps^{-\xi}\left(\frac{r_k}{r_j}\right)^{d/2-s}
            \|v_j\|_{L^2(\Omega_j)}
            =C\eps^{-\xi}\kappa_{jk}\|v_j\|_{L^2(\Omega_j)}.
        \end{aligned}
    \end{equation}
    Here we used $r_{j\wedge i}\geq r_j$ and $s<d$ to obtain
    \[
        \sum_{i=0}^{\infty}r_i^d r_{j\wedge i}^{s-d}r_{k\wedge i}^{-d-s}
        \leq r_j^{s-d}r_k^{-s}\sum_{i=0}^{\infty}2^{-s|i-k|}.
    \]

    \fbox{Case 3-2: $\delta(x_0)>8r_k$.}
    Set $R_0:=\frac12(\delta(x_0)\wedge r_j)>4r_k$. Then
    $B_{R_0}(x_0)\subset\Omega$ and
    $B_{R_0}(x_0)\cap\supp v_j=\varnothing$.
    Let $W_j$ solve \eqref{eqn:local-source-annuli} with source $|v_j|$.
    By comparison, $W_j\geq|w_j|$ in $\RR^d$, and both functions are
    $s$-harmonic in $B_{R_0}(x_0)$. Theorem~\ref{thm:Green-estimate} gives
    \[
        W_j(x_0)
        \leq C r_j^{2s-d}\left(1\wedge\frac{\delta(x_0)}{r_j}\right)^s
        \|v_j\|_{L^1(\Omega_j)}
        \leq C R_0^s r_j^{s-d/2}\|v_j\|_{L^2(\Omega_j)}.
    \]
    Since $s<1$, Lemma~\ref{lem:gradient-estimate} (interior gradient estimate) yields
    \begin{equation}
        \|\nabla w_j\|_{L^\infty(B_{R_0/2}(x_0))}
        \leq C R_0^{-1}W_j(x_0)
        \leq C r_k^{s-1}r_j^{s-d/2}\|v_j\|_{L^2(\Omega_j)}.
    \label{eqn:gradient-scaled}
    \end{equation}
    Since $A_k\subset B_{R_0/2}(x_0)$, the definition of the seminorm and
    $c_{d,t}/(1-t)\leq C$ give, uniformly for $0<t<1$,
    \[
        |w_j|_{H^t(A_k)}
        \leq C r_k^{d/2+1-t}
        \|\nabla w_j\|_{L^\infty(B_{R_0/2}(x_0))}.
    \]
    Taking $t=2s-\eps$ and using \eqref{eqn:gradient-scaled}, we obtain
    \begin{equation}
        |w_j|_{H^{2s-\eps}(A_k)}
        \leq C r_k^\eps\left(\frac{r_k}{r_j}\right)^{d/2-s}
        \|v_j\|_{L^2(\Omega_j)}
        \leq C \kappa_{jk}\|v_j\|_{L^2(\Omega_j)}.
    \end{equation}

    For $A_\ast$, use the core estimate in \eqref{eqn:local-sobolev-shift}
    at scale $r_J$. If $J\leq j+3$, the argument in Case~2 applies.
    If $J\geq j+4$ and $\delta(x_0)\leq8r_J$, the tail is estimated as in
    Case~3-1. For the local $L^2$ term, squaring and summing the bounds in
    Case~3-1 with $k=J$ over $i\geq J-2$ gives
    \[
        r_J^{-2s}\|w_j\|_{L^2(\Omega_\ast'')}
        \leq C r_J^{d/2-s}r_j^{s-d/2}\|v_j\|_{L^2(\Omega_j)}
        =C\kappa_{jJ}\|v_j\|_{L^2(\Omega_j)}.
    \]
    If $J\geq j+4$ and $\delta(x_0)>8r_J$, Case~3-2 applies with
    $A_k$ replaced by $B_{r_J}(x_0)$.
    In each case, we obtain \eqref{eqn:local-regularity-annuli-core}, completing the proof.
    \end{proof}

    \begin{proof}[Proof of Lemma~\ref{lem:local-L2-core} when $2s<d$]
    Recall that $r_J\geq2h$, so $r_k\geq4h$ for $k<J$.
    Since $T_0\subset\overline{B_h(x_0)}$,
    \eqref{eq:L2-out-comp-core} follows as in the proof of
    Lemma~\ref{lem:local-L2}, with the source scale $r_j$ replaced by $h$.

    For the core, set $J_c:=\lfloor-\log_2(2h)\rfloor$ and
    $\Omega_{\ast,c}:=\Omega\cap B_{r_{J_c}}(x_0)$. Then
    $J\leq J_c$, $2h\leq r_{J_c}<4h$, and
    \[
        \Omega_\ast=\bigcup_{k=J}^{J_c-1}\Omega_k\cup\Omega_{\ast,c}.
    \]
    Since $\bigl\||\cdot|^{2s-d}\chi_{B_{2r_{J_c}}(0)}\bigr\|_{L^1(\RR^d)}
    \leq C r_{J_c}^{2s}$, Young's inequality gives
    \[
        \|z\|_{L^2(\Omega_{\ast,c})}
        \leq C r_{J_c}^{2s}\|\varphi\|_{L^2(T_0)}
        \leq C h^{2s}\|\varphi\|_{L^2(T_0)}.
    \]
    Combining this with \eqref{eq:L2-out-comp-core}, applied with terminal
    index $J_c$, gives
    \[
    \begin{aligned}
        \|z\|_{L^2(\Omega_\ast)}^2
        &=\sum_{k=J}^{J_c-1}\|z\|_{L^2(\Omega_k)}^2
          +\|z\|_{L^2(\Omega_{\ast,c})}^2\\
        &\leq C\biggl(h^d\sum_{k=J}^{J_c-1}r_k^{4s-d}+h^{4s}\biggr)
          \|\varphi\|_{L^2(T_0)}^2.
    \end{aligned}
    \]
    Evaluating the sum for $d>4s$, $d=4s$, and $2s<d<4s$ proves
    \eqref{eq:L2-core-core}.
    \end{proof}

    \begin{proof}[Proof of Lemma~\ref{lem:local-regularity-core} when $2s<d$]
    Set $J_c:=\lfloor-\log_2(2h)\rfloor$. Since
    $T_0\subset\overline{B_h(x_0)}\subset\overline{B_{r_{J_c+1}}(x_0)}$,
    we may write $\varphi=\sum_{j\geq J_c+1}\varphi_j$ a.e., where
    $\varphi_j:=\varphi\chi_{\Omega_j}$. Let $w_j$ solve
    \eqref{eqn:local-source-annuli} with $v_j=\varphi_j$.
    By linearity and \eqref{eqn:Sobolev-shift-all},
    $z=\sum_{j\geq J_c+1}w_j$ in $H^{2s-\eps}(\RR^d)$.
    For each $j$, apply Lemma~\ref{lem:local-regularity}
    (annular sources: Sobolev estimates) with a terminal index larger than $j$.
    Recall that $r_J\geq2h$ implies $J\leq J_c$. For $k<J$ and
    $j\geq J_c+1$, we therefore have $k\leq j-2$ and
    $\kappa_{jk}\leq C(r_j/r_k)^{d/2}$. Cauchy--Schwarz then gives
    \[
    \begin{aligned}
    |z|_{H^{2s-\eps}(A_k)}
    &\leq C\eps^{-\xi}r_k^{-d/2}
      \sum_{j\geq J_c+1}r_j^{d/2}\|\varphi_j\|_{L^2(\Omega_j)}\\
    &\leq C\eps^{-\xi}r_k^{-d/2}
      \biggl(\sum_{j\geq J_c+1}r_j^d\biggr)^{1/2}\|\varphi\|_{L^2(T_0)}\\
    &\leq C\eps^{-\xi}(h/r_k)^{d/2}\|\varphi\|_{L^2(T_0)}.
    \end{aligned}
    \]
    The core estimate follows directly from \eqref{eqn:Sobolev-shift-all}.
    This proves the assertion.
    \end{proof}

\section{Boundary-corrected interpolation estimates}\label{sec:interpolation}
Fractional stability and approximation estimates for Scott--Zhang
interpolation \cite{ScottZhang1990} are available in
\cite{Ciarlet2013,AcostaBorthagaray2017}. Here we use element averages
and set boundary nodal values to zero on $\partial\Omega_h$.
We account for the boundary correction and the error on
$\Omega\setminus\Omega_h$ when $\Omega_h\subsetneq\Omega$, and derive
local estimates, including the boundary-weighted
term, on balls and annuli for the subsequent duality arguments.

Let $\mc N_h$ and $\mc N_h^\partial$ denote the vertices of $\mc T_h$
and those on $\partial\Omega_h$, respectively. For each $a\in\mc N_h$,
choose an element $S_a$ containing $a$ and its dual function
$\psi_a\in\mathbb P_1(S_a)$, characterized by
$\int_{S_a}\psi_a q\,\mathrm{d}x=q(a)$ for $q\in\mathbb P_1(S_a)$.
Set $\ell_a(v):=\int_{S_a}\psi_a v\,\mathrm{d}x$ and define
\begin{equation}\label{eq:Ih}
 J_hv:=\sum_{a\in\mc N_h}\ell_a(v)\varphi_a,
 \qquad
 I_hv:=\sum_{a\in\mc N_h\setminus\mc N_h^\partial}
          \ell_a(v)\varphi_a,
\end{equation}
where $\varphi_a\in\widehat V_h$ is the nodal basis function at $a$.
Thus $I_h:L^2(\Omega)\to V_h$ is local and reproduces affine functions
on elements whose vertices are interior nodes. For $T\in\mc T_h$,
let $\omega_T$ be the union of all elements sharing a vertex with $T$.

\begin{lemma}[boundary covering]\label{lem:caps-overlaps}
For all sufficiently small $h$ and every element $T$ with
$\omega_T\cap\partial\Omega_h\ne\varnothing$, there is a ball
$B_T\supset\omega_T$ such that
\begin{equation}\label{eqn:cap-properties}
 \diam B_T\le C(h+\delta_h),
 \qquad |B_T\cap\Omega^c|\geq C^{-1}(h+\delta_h)^d,
\end{equation}
\begin{equation}\label{eqn:cap-covering}
    \Omega\setminus\Omega_h \subset \bigcup_{T:\omega_T\cap\partial\Omega_h\neq\varnothing} B_T. 
\end{equation}
The balls have overlap multiplicity at most
\begin{equation}\label{eqn:geometric-overlap}
 M_h\le C\left(1+\frac{\delta_h}{h}\right)^d.
\end{equation}
Here $C=C(d,\gamma,C_{\mathrm{geo}},c_\Omega,r_\Omega)$.
\end{lemma}
\begin{proof}
Choose $z_T\in\omega_T\cap\partial\Omega_h$ and $y_T\in\partial\Omega$
with $|z_T-y_T|\leq\delta_h$, and set
$B_T:=B_{C(h+\delta_h)}(y_T)$. For sufficiently large $C$, this ball
contains $\omega_T$; the exterior density condition
\eqref{eqn:exterior-density} gives \eqref{eqn:cap-properties}.
For $x\in\Omega\setminus\Omega_h$, the mesh-domain assumptions provide
$y\in\partial\Omega$ and $z\in\partial\Omega_h$ with
$|x-y|\leq C_{\mathrm{geo}}\delta_h$ and $|y-z|\leq\delta_h$.
Choosing $T$ with $z\in\omega_T$ gives $|x-y_T|\leq C(h+\delta_h)$,
which proves \eqref{eqn:cap-covering}.
Finally, $x\in B_T$ implies $T\subset B_{C(h+\delta_h)}(x)$.
Disjoint element interiors and quasi-uniformity give
$\#\{T:x\in B_T\}\leq C(h+\delta_h)^d/h^d$, proving
\eqref{eqn:geometric-overlap}.
\end{proof}

\subsection{Poincar\'e inequalities}
We recall the following standard estimates without proof.
Here $\mathbb P_m$ denotes polynomials of total degree at most $m$.

\begin{lemma}[integer-order Poincar\'e inequalities]\label{lem:H1-Poincare}
Let $B:=B_r(x_0)$, $m\in\{1,2\}$ and $w\in H^m(B)$. Then
\begin{equation}\label{eqn:integer-Poincare-mod-polynomials}
 \inf_{p\in\mathbb P_{m-1}}
 \sum_{j=0}^{m-1}r^j|w-p|_{H^j(B)}
 \leq Cr^m|w|_{H^m(B)},
\end{equation}
with $C=C(d)$. If $w=0$ almost everywhere on a measurable set $E\subset B$ with
$|E|\geq\kappa|B|>0$, the same bound holds with $p=0$ and $C=C(d,\kappa)$.
\end{lemma}

For $0<t\leq2$, write
\begin{equation}\label{eqn:interpolation-loss}
 \Lambda(t):=\begin{cases}
 t^{-1/2},&0<t\leq1,\\
 (t-1)^{-1/2},&1<t\leq2.
 \end{cases}
\end{equation}
The following fractional version uses the normalization \eqref{eq:slobo};
for $0<t<1$, see \cite[Proposition~2.4]{BianchiBrasco2022}.
For $1<t<2$, the estimates follow by combining the order-$(t-1)$
estimates for first derivatives with the integer-order Poincar\'e
inequality.

\begin{lemma}[fractional Poincar\'e inequalities]\label{lem:frac-Poincare}
Let $B:=B_r(x_0)$, $0<t<2$, $t\ne1$, and $m:=\lceil t\rceil$.
For $w\in H^t(B)$,
\begin{equation}\label{eqn:fractional-Poincare-mod-polynomials}
 \inf_{p\in\mathbb P_{m-1}}
 \sum_{j=0}^{m-1}r^{j-t}|w-p|_{H^j(B)}
 \leq C\Lambda(t)|w|_{H^t(B)},
\end{equation}
with $C=C(d)$. If $w=0$ almost everywhere on a measurable set $E\subset B$ with
$|E|\geq\kappa|B|>0$, the same bound holds with $p=0$ and $C=C(d,\kappa)$.
The constants are independent of $t,r,x_0$ and $E$; the integer endpoints
are covered by Lemma~\ref{lem:H1-Poincare}.
\end{lemma}

\subsection{Global interpolation estimates}

We retain the dependence on
$\delta_h/h$ explicitly. Constants may depend on $d$, the Lipschitz
character of $\Omega$, $c_\Omega$, $r_\Omega$, $C_{\mathrm{geo}}$ and $\gamma$.

\begin{lemma}[global integer-order interpolation estimates]\label{lem:global-endpoints}
The interpolation operator $I_h$ satisfies
\begin{equation}\label{eqn:L2-stability}
 \|I_hv\|_{L^2(\RR^d)}\le C\|v\|_{L^2(\Omega)},
 \qquad v\in L^2(\Omega),
\end{equation}
and
\begin{equation}\label{eqn:H1-stability}
 \|I_hv\|_{H^1(\RR^d)}
 \le C\left(1+\frac{\delta_h}{h}\right)^{1+d/2}
 \|v\|_{H^1(\RR^d)},
 \qquad v\in\widetilde H^1(\Omega)=H_0^1(\Omega).
\end{equation}
For $m\in\{1,2\}$ and $v\in\widetilde H^m(\Omega)=H_0^m(\Omega)$,
\begin{equation}\label{eqn:global-integer-approx}
 \|v-I_hv\|_{L^2(\RR^d)}+h|v-I_hv|_{H^1(\RR^d)}
 \le Ch^m\left(1+\frac{\delta_h}{h}\right)^{m+d/2}
 |v|_{H^m(\RR^d)}.
\end{equation}
The constants $C=C(d,\gamma,C_{\mathrm{geo}},c_\Omega,r_\Omega)$ are independent of $h$ and $\delta_h$.
\end{lemma}

\begin{proof}
Elementwise scaling gives the standard local bound
\begin{equation}\label{eqn:local-integer-stability}
 \|I_hv\|_{L^2(T)}+h\|\nabla I_hv\|_{L^2(T)}
 \leq C\|v\|_{L^2(\omega_T)}.
\end{equation}
Bounded patch overlap yields \eqref{eqn:L2-stability}.
For $m\in\{1,2\}$ and $\omega_T\cap\partial\Omega_h=\varnothing$,
$I_h=J_h$ on $T$, and the standard interior estimate
\cite[Section~4]{ScottZhang1990} gives
\[
 \|v-I_hv\|_{L^2(T)}+h|v-I_hv|_{H^1(T)}
 \leq Ch^m|v|_{H^m(\omega_T)}.
\]
If $\omega_T\cap\partial\Omega_h\ne\varnothing$, the ball $B_T$ from
Lemma~\ref{lem:caps-overlaps} contains a fixed proportion of $\Omega^c$,
where $v=0$. Lemma~\ref{lem:H1-Poincare} and
\eqref{eqn:local-integer-stability} therefore give
\[
 \begin{aligned}
 \|v-I_hv\|_{L^2(T)}+h|v-I_hv|_{H^1(T)}
 &\leq C\bigl(\|v\|_{L^2(B_T)}+h|v|_{H^1(B_T)}\bigr)\\
 &\leq C(h+\delta_h)^m|v|_{H^m(B_T)}.
 \end{aligned}
\]
On $\Omega\setminus\Omega_h$, the error equals $v$. The same Poincar\'e
estimate and the covering \eqref{eqn:cap-covering} yield
\[
 \|v\|_{L^2(\Omega\setminus\Omega_h)}^2
 +h^2\|\nabla v\|_{L^2(\Omega\setminus\Omega_h)}^2
 \leq C(h+\delta_h)^{2m}
 \sum_{T:\,\omega_T\cap\partial\Omega_h\ne\varnothing}|v|_{H^m(B_T)}^2.
\]
Summation with the overlap bound \eqref{eqn:geometric-overlap} proves
\eqref{eqn:global-integer-approx}. Its case $m=1$, together with
\eqref{eqn:L2-stability}, gives \eqref{eqn:H1-stability}.
\end{proof}

\begin{lemma}[global fractional-order interpolation estimates]\label{lem:global-fractional}
For $0<s<1$ and $v\in\widetilde H^s(\Omega)$,
\begin{equation}\label{eqn:Ht-stability}
 \|I_hv\|_{H^s(\RR^d)}
 \le C\left(1+\frac{\delta_h}{h}\right)^{s(1+d/2)}
 \|v\|_{H^s(\RR^d)}.
\end{equation}
For $0<s<1$, $s\le t\le2$, and $v\in\widetilde H^t(\Omega)$,
\begin{equation}\label{eq:Ht-global-approx}
 \begin{aligned}
 &|v-I_hv|_{H^s(\RR^d)}
  +h^{-s}\|v-I_hv\|_{L^2(\RR^d)} \le Ch^{t-s}
 \left(1+\frac{\delta_h}{h}\right)^{t(1+d/2)}
 \|v\|_{H^t(\RR^d)}.
 \end{aligned}
\end{equation}
The constants $C=C(d,\gamma,C_{\mathrm{geo}},c_\Omega,r_\Omega,\mathrm{Lip}_\Omega)$ are independent of $s$, $t$, $h$, and $\delta_h$, where
$\mathrm{Lip}_\Omega$ is the Lipschitz character of $\Omega$.
\end{lemma}

\begin{proof}
Apply normalized real interpolation to the global integer-order
interpolation estimates in Lemma~\ref{lem:global-endpoints}; see
\cite[Section~2.1]{BorthagarayLeykekhmanNochetto2021} for the Sobolev
norm conventions. Since $\Omega$ is Lipschitz, the zero-extension
spaces $\widetilde H^t(\Omega)$ form an interpolation scale, with
norm-equivalence constants uniform for orders in $[0,2]$; see
\cite[Theorem~4.6 and Corollary~4.10]{ChandlerWildeHewettMoiola2015}.
Interpolation of \eqref{eqn:L2-stability} and
\eqref{eqn:H1-stability} gives \eqref{eqn:Ht-stability}.
For $m=0,1,2$ and $v\in\widetilde H^m(\Omega)$, we have
\[
 \|v-I_hv\|_{L^2(\RR^d)}
 \leq Ch^m\left(1+\frac{\delta_h}{h}\right)^{m(1+d/2)}
       \|v\|_{H^m(\RR^d)}.
\]
The case $m=0$ follows from $L^2$ stability and the triangle inequality;
the cases $m=1,2$ follow from \eqref{eqn:global-integer-approx}, using
$m+d/2\leq m(1+d/2)$.
Interpolating in the source spaces between orders $0$ and $1$, and
between $1$ and $2$, gives
\[
 \|v-I_hv\|_{L^2(\RR^d)}
 \leq Ch^t\left(1+\frac{\delta_h}{h}\right)^{t(1+d/2)}
       \|v\|_{H^t(\RR^d)},\qquad 0\leq t\leq2.
\]

Interpolating the target norms in \eqref{eqn:global-integer-approx}
gives error bounds from $\widetilde H^m(\Omega)$ to $H^s(\RR^d)$
for $m=1,2$.
For $s\leq t\leq1$, interpolate the $m=1$ bound with the $H^s$
stability of $I-I_h$ from \eqref{eqn:Ht-stability}.
For $1\leq t\leq2$, interpolate the $m=1,2$ bounds in the source spaces.
Together with the preceding $L^2$ estimate, this proves
\eqref{eq:Ht-global-approx}, with constants uniform in $s$ and $t$.
\end{proof}

\subsection{Local interpolation estimates}
Throughout this subsection, we use the assumption
$\delta_h\leq C_{\mathrm{geo}}h$ and absorb the factors involving
$\delta_h/h$ into constants depending on $C_{\mathrm{geo}}$, uniformly
for $0<t\leq2$.
We first prove the estimate on concentric balls.

\begin{lemma}[local interpolation on balls]\label{lem:local-fractional}
Let $q>1$, $x_0\in\RR^d$, and suppose that
\[
 \max\left\{2,\frac{8(1+C_{\mathrm{geo}})^2}{q-1}\right\}h
 \leq r\leq1.
\]
Write $B_r:=B_r(x_0)$ and $B_{qr}:=B_{qr}(x_0)$.
For $0<s<1$, $s\leq t\leq2$, and $v\in\widetilde H^t(\Omega)$,
\begin{equation}\label{eqn:local-fractional}
 \begin{aligned}
 &|v-I_hv|_{H^s(B_r)}
  +c_{d,s}^{1/2}\|(v-I_hv)\delta^{-s}\|_{L^2(B_r\cap\Omega)}
  +h^{-s}\|v-I_hv\|_{L^2(B_r)}\\
 &\qquad\leq C\Lambda(t)h^{t-s}|v|_{H^t(B_{qr})},
 \end{aligned}
\end{equation}
where $\Lambda(t)$ is defined in \eqref{eqn:interpolation-loss}.
The constant $C=C(d,q,\gamma,C_{\mathrm{geo}},c_\Omega,r_\Omega,\mathrm{Lip}_\Omega)$
is independent of $s$, $t$, $h$, $r$, $x_0$, and $v$.
\end{lemma}

\begin{proof}
We construct a localized function with the same interpolation error on
$B_r$, to which the global estimates of
Lemma~\ref{lem:global-fractional} can be applied.
Put $q_0:=(1+q)/2$. Every element meeting $B_r$, together with all its
vertex patches, is contained in $B_{r+2h}(x_0)$. The radius condition
therefore permits a cutoff $\eta\in C_0^\infty(B_{q_0r}(x_0))$ such that
\[
 0\leq\eta\leq1,\qquad \eta=1\text{ on }B_{r+2h}(x_0),\qquad
 \|\nabla^i\eta\|_{L^\infty}\leq C_qr^{-i}\quad(i=1,2),
\]
with $\dist(\supp\eta,B_{q_0r}(x_0)^c)\geq c_qr>0$.
Thus $\eta=1$ on all averaging elements used to define $I_hv$ on $B_r$.

If $B_{q_0r}(x_0)\subset\Omega_h$, choose a constant $p$ for $0<t\leq1$
and an affine polynomial $p$ for $1<t\leq2$ using
Lemmas~\ref{lem:H1-Poincare} and~\ref{lem:frac-Poincare}
(Poincar\'e inequalities) on this ball. All vertices of elements meeting
$B_r$ are interior nodes, so $I_hp=p$ there. Applying
Lemma~\ref{lem:frac-cutoff} (cutoff estimate) together with
Lemmas~\ref{lem:H1-Poincare} and~\ref{lem:frac-Poincare}
(Poincar\'e inequalities) gives
\[
 \|\eta(v-p)\|_{H^t(\RR^d)}
 \leq C\Lambda(t)|v|_{H^t(B_{q_0r}(x_0))}
 \leq C\Lambda(t)|v|_{H^t(B_{qr})}.
\]
Here the Poincar\'e inequalities control the lower-order terms arising
in the cutoff estimate; since $r\leq1$, they also control the full
$H^t$ norm by the same seminorm.

Otherwise, $B_{q_0r}(x_0)\cap\Omega_h^c \neq\varnothing$.
If $B_{q_0r}(x_0)\cap\Omega=\varnothing$, the error vanishes on $B_r$
and there is nothing to prove. Otherwise, we show that $\Omega^c$
occupies a fixed proportion of $B_{qr}$, so that the zero-set
Poincar\'e inequalities apply. The connectedness of the ball and
$\Omega_h\subset\Omega$ allow us to
choose $z\in B_{q_0r}(x_0)\cap(\overline\Omega\setminus\Omega_h)$.
The mesh-domain condition \eqref{eqn:mesh-domain-distance} gives
$\dist(z,\partial\Omega)\leq C_{\mathrm{geo}}\delta_h$; this also holds
when $z\in\partial\Omega$, since the distance is then zero.
Hence there exists $y\in\partial\Omega$ with
\[
 |y-x_0|<q_0r+C_{\mathrm{geo}}\delta_h
 \leq q_0r+\frac{q-1}{8}r,
 \qquad B_{(q-1)r/4}(y)\subset B_{qr}.
\]
The exterior density condition \eqref{eqn:exterior-density} and $r\leq1$ imply
\[
 |B_{qr}\cap\Omega^c|
 \geq c_\Omega\min\{(q-1)r/4,r_\Omega\}^d
 \geq c_\Omega\min\{(q-1)/4,r_\Omega\}^d r^d.
\]
Since $v=0$ on this set, the zero-set versions of
Lemmas~\ref{lem:H1-Poincare} and~\ref{lem:frac-Poincare}
(Poincar\'e inequalities), now on $B_{qr}$, combined with
Lemma~\ref{lem:frac-cutoff} (cutoff estimate), yield
\[
 \|\eta v\|_{H^t(\RR^d)}\leq C\Lambda(t)|v|_{H^t(B_{qr})}.
\]
Set $p=0$ in this case.

In both cases $w:=\eta(v-p)$ belongs to $\widetilde H^t(\Omega)$.
Since $\eta=1$ on the averaging elements and $I_hp=p$ on $B_r$
(trivially when $p=0$),
\[
 w-I_hw=(v-p)-I_h(v-p)=v-I_hv\qquad\text{on }B_r.
\]
To control the boundary-weighted term, use $w-I_hw=0$ in $\Omega^c$
and the exterior contribution to its fractional energy:
\[
 \begin{aligned}
 |w-I_hw|_{H^s(\RR^d)}^2
 &\geq c_{d,s}\int_{B_r\cap\Omega}|v-I_hv|^2
       \int_{\Omega^c}|x-y|^{-d-2s}\,\mathrm{d}y\mathrm{d}x\\
 &\geq c\,c_{d,s}\|(v-I_hv)\delta^{-s}\|_{L^2(B_r\cap\Omega)}^2.
 \end{aligned}
\]
The second inequality follows from the lower bound
\eqref{eqn:delta-upper-bound} in Lemma~\ref{lem:delta-upper-bound}
(exterior kernel bound). Since $\diam\Omega\leq1$, its proof allows
$c=c_\Omega2^{-d-2}\min\{1,r_\Omega\}^d$, uniformly for $0<s<1$.
We now apply Lemma~\ref{lem:global-fractional} (global fractional-order
interpolation estimates) to $w$, using $\delta_h/h\leq C_{\mathrm{geo}}$ to absorb
the geometric factor. Together with restriction to $B_r$ and the
preceding bounds, this gives
\[
 \begin{aligned}
 &|v-I_hv|_{H^s(B_r)}
  +c_{d,s}^{1/2}\|(v-I_hv)\delta^{-s}\|_{L^2(B_r\cap\Omega)}
  +h^{-s}\|v-I_hv\|_{L^2(B_r)}\\
 &\qquad\leq C\bigl(|w-I_hw|_{H^s(\RR^d)}
                  +h^{-s}\|w-I_hw\|_{L^2(\RR^d)}\bigr)\\
 &\qquad\leq Ch^{t-s}\|w\|_{H^t(\RR^d)}\leq C\Lambda(t)h^{t-s}|v|_{H^t(B_{qr})},
 \end{aligned}
\]
which proves \eqref{eqn:local-fractional}.
\end{proof}

A finite ball covering gives the corresponding annular estimates.
\begin{corollary}[annular and core interpolation estimates]\label{cor:local-interp-estimate}
Suppose $r_J\geq64(1+C_{\mathrm{geo}})^2h$.
For $0\leq k<J$, $0<s<1$, $s\leq t\leq2$, and
$v\in\widetilde H^t(\Omega)$,
\begin{equation}\label{eqn:local-interp}
 \begin{aligned}
 &|v-I_hv|_{H^s(A_k)}
  +c_{d,s}^{1/2}\|(v-I_hv)\delta^{-s}\|_{L^2(\Omega_k)}
  +h^{-s}\|v-I_hv\|_{L^2(\Omega_k)}\\
 &\qquad\leq C\Lambda(t)h^{t-s}|v|_{H^t(A_k')},
 \end{aligned}
\end{equation}
\begin{equation}\label{eqn:local-interp-core}
 \begin{aligned}
 &|v-I_hv|_{H^s(A_\ast)}
  +c_{d,s}^{1/2}\|(v-I_hv)\delta^{-s}\|_{L^2(\Omega_\ast)}
  +h^{-s}\|v-I_hv\|_{L^2(\Omega_\ast)}\\
 &\qquad\leq C\Lambda(t)h^{t-s}|v|_{H^t(A_\ast')}.
 \end{aligned}
\end{equation}
Here $\Lambda(t)$ is defined in \eqref{eqn:interpolation-loss}.
The constant $C=C(d,\gamma,C_{\mathrm{geo}},c_\Omega,r_\Omega,\mathrm{Lip}_\Omega)$
is independent of $s,t,h,x_0,k,J,v$.
The same estimates hold with $(A_k,A_k')$ replaced by
$(A_k^{(p)},A_k^{(p+1)})$ and $\Omega_k$ by $\Omega_k^{(p)}$,
and likewise for the core regions, provided
$r_{J+p}\geq64(1+C_{\mathrm{geo}})^2h$, with constants depending
additionally on $p\in\NN_+$.
\end{corollary}

\begin{proof}
Set $\tilde r_k:=r_k/16$ and cover $\overline{A_k}$ by balls
$B_{\tilde r_k}(x_i)$, $x_i\in\overline{A_k}$, $1\leq i\leq N$, with $N\leq C_d$.
For every $i$,
\[
 B_{4\tilde r_k}(x_i)\subset
 \{x:r_k/4<|x-x_0|<5r_k/4\}\subset A_k'.
\]
Apply Lemma~\ref{lem:local-fractional} (local interpolation on balls)
with $r=2\tilde r_k$ and $q=2$.
Its radius condition follows from the assumed bound on $r_J$.
Summing the squared estimates controls the weighted term and
$h^{-s}\|v-I_hv\|_{L^2(A_k)}$ by the right-hand side of
\eqref{eqn:local-interp}. For the fractional seminorm, interactions
not captured by the local seminorms are controlled by
Lemma~\ref{lem:frac-localize} (localization), which gives
\[
 \begin{aligned}
 |v-I_hv|_{H^s(A_k)}^2
 &\leq\sum_{i=1}^N|v-I_hv|_{H^s(B_{2\tilde r_k}(x_i))}^2
       +C\tilde r_k^{-2s}\|v-I_hv\|_{L^2(A_k)}^2\\
 &\leq C\Lambda(t)^2h^{2(t-s)}|v|_{H^t(A_k')}^2,
 \end{aligned}
\]
since $h\leq\tilde r_k$ and $N\leq C_d$. This proves
\eqref{eqn:local-interp}. For the core, apply
Lemma~\ref{lem:local-fractional} (local interpolation on balls) directly
with $r=r_J$ and $q=2$.
For $A_k^{(p)}$, repeat the covering with $\tilde r_{k+p}:=r_{k+p}/16$;
the balls $B_{4\tilde r_{k+p}}(x_i)$ lie in $A_k^{(p+1)}$ and their number is
bounded in terms of $d,p$. For $A_\ast^{(p)}$, use
$\tilde r_{J+p}:=r_{J+p}/16$ in the same covering argument. Its enlarged balls lie
in $A_\ast^{(p+1)}$, and the stronger radius condition ensures that
Lemma~\ref{lem:local-fractional} (local interpolation on balls) applies
to every covering ball.
\end{proof}

Combining Corollary~\ref{cor:local-interp-estimate} with the local regularity
estimates in Section~\ref{sec:regularity} gives the following bounds for
the two dual problems.
\begin{corollary}[interpolation estimates for dual solutions]\label{cor:dual-local-estimate}
Let $0<s\leq1/2$ and
$r_J\geq64(1+C_{\mathrm{geo}})^2h$.
For $w_j$ defined in \eqref{eqn:local-source-annuli} and $0\leq j,k<J$,
\begin{equation}\label{eqn:dual-local-interp}
 \begin{aligned}
 &|w_j-I_hw_j|_{H^s(A_k)}
  +c_{d,s}^{1/2}\|(w_j-I_hw_j)\delta^{-s}\|_{L^2(\Omega_k)}
  +h^{-s}\|w_j-I_hw_j\|_{L^2(\Omega_k)}\\
 &\qquad\leq Ch^s\ell_h^\xi\kappa_{jk}\|v_j\|_{L^2(\Omega_j)},
 \end{aligned}
\end{equation}
\begin{equation}\label{eqn:dual-local-interp-core}
 \begin{aligned}
 &|w_j-I_hw_j|_{H^s(A_\ast)}
  +c_{d,s}^{1/2}\|(w_j-I_hw_j)\delta^{-s}\|_{L^2(\Omega_\ast)}
  +h^{-s}\|w_j-I_hw_j\|_{L^2(\Omega_\ast)}\\
 &\qquad\leq Ch^s\ell_h^\xi\kappa_{jJ}\|v_j\|_{L^2(\Omega_j)}.
 \end{aligned}
\end{equation}
For $z$ defined in \eqref{eqn:local-source} and $0\leq k<J$,
\begin{equation}\label{eqn:dual-core-local-interp}
 \begin{aligned}
 &|z-I_hz|_{H^s(A_k)}
  +c_{d,s}^{1/2}\|(z-I_hz)\delta^{-s}\|_{L^2(\Omega_k)}
  +h^{-s}\|z-I_hz\|_{L^2(\Omega_k)}\\
 &\qquad\leq Ch^s\ell_h^\xi(h/r_k)^{d/2}\|\varphi\|_{L^2(T_0)},
 \end{aligned}
\end{equation}
\begin{equation}\label{eqn:dual-core-local-interp-core}
 \begin{aligned}
 &|z-I_hz|_{H^s(A_\ast)}
  +c_{d,s}^{1/2}\|(z-I_hz)\delta^{-s}\|_{L^2(\Omega_\ast)}
  +h^{-s}\|z-I_hz\|_{L^2(\Omega_\ast)}\\
 &\qquad\leq Ch^s\ell_h^\xi\|\varphi\|_{L^2(T_0)}.
 \end{aligned}
\end{equation}
Here $\xi$ and $\kappa_{jk}$ are defined in \eqref{eqn:xi-alpha} and
\eqref{eqn:kappa}, respectively.
The constant $C=C(d,s,\Omega,\gamma,C_{\mathrm{geo}})$ is independent
of $h,x_0,j,k,J,v_j,\varphi$; its domain dependence is as in the preceding
interpolation and local regularity estimates.
The estimates also hold with $(A_k,\Omega_k)$ replaced by
$(A_k^{(p)},\Omega_k^{(p)})$, and likewise for the core regions, for each
fixed $p\in\NN_+$, provided $r_{J+p}\geq64(1+C_{\mathrm{geo}})^2h$,
with constants depending additionally on $p$.
\end{corollary}

\begin{proof}
Take $\eps:=\min\{s/2,\ell_h^{-1}\}$ and $t:=2s-\eps$.
Theorem~\ref{thm:sobolev-shift-lip} (global Sobolev regularity) gives
$w_j,z\in\widetilde H^t(\Omega)$. Since $s\leq t<1$,
$\Lambda(t)\leq s^{-1/2}$. Together with $h^{-\eps}\leq e$ and
$\eps^{-\xi}\leq C\ell_h^\xi$, this yields
\[
 \Lambda(t)h^{s-\eps}\eps^{-\xi}\leq Ch^s\ell_h^\xi.
\]
Applying
Corollary~\ref{cor:local-interp-estimate} (annular and core interpolation
estimates) and using the enlarged-region versions of
Lemma~\ref{lem:local-regularity} (annular sources: Sobolev estimates) and
Lemma~\ref{lem:local-regularity-core} (element sources: Sobolev estimates),
we obtain all four estimates and their enlarged-region versions from
the preceding bound.
\end{proof}

\section{Annular error estimates for the dual problem}\label{sec:annular-error}

For the annular error analysis, we assume $0<s\leq1/2$ and $d\geq1$, including the critical
case $(s,d)=(1/2,1)$. The regularity estimates for this case are proved in
Appendix~\ref{app:critical-case} and enter the analysis through
Corollary~\ref{cor:dual-local-estimate}.

Fix $x_0\in T_0\in\mc T_h$ and $\varphi\in C_0^\infty(T_0)$ with
$\|\varphi\|_{L^2(T_0)}=1$. Let $z$ solve \eqref{eqn:local-source} and
set $z_h:=R_h^s z$. We use the dyadic decomposition from
Section~\ref{sec:regularity} with terminal scale
\begin{equation}\label{eqn:terminal-scale}
 H:=C_\ast h\ell_h^\beta,\qquad J:=\lceil-\log_2 H\rceil,
\end{equation}
where $C_\ast\geq1$ and $\beta\geq0$ will be chosen so that the local
interpolation estimates apply and the factor $(h/H)^s\ell_h^\xi$ in the
final absorption argument is small. Here $\xi$ is defined in
\eqref{eqn:xi-alpha}.
For fixed $C_\ast,\beta$ and sufficiently small $h$, we have $0<H<1$,
$H/2<r_J\leq H$, and $r_j\geq H$ for $0\leq j<J$.
Since $r_{J+p}>2^{-p-1}H$, taking $C_\ast$ sufficiently large ensures
the mesh-scale restrictions in Corollaries~\ref{cor:local-interp-estimate}
and~\ref{cor:dual-local-estimate} for all fixed enlargement indices used below.

Define the scaled energy by
\begin{equation}\label{eqn:scale-energy-def}
 \mE_j(v):=|v|_{H^s(A_j')}+r_j^{-s}\|v\|_{L^2(A_j')},
 \quad 0\leq j<J.
\end{equation}

\begin{lemma}[annular energy error]\label{lem:annular-energy-error}
Let $\beta=\xi/s$ in \eqref{eqn:terminal-scale}. There exists
$C_0=C_0(d,s,\gamma,\Omega,C_{\mathrm{geo}})>0$ such that, for
$C_\ast\geq C_0$ and sufficiently small $h$,
\begin{equation}\label{eqn:annular-energy-error}
 \mE_j(z-z_h)\leq C r_j^{-d/2}h^{d/2+s}\ell_h^{d\xi/(2s)+\xi},
 \quad 0\leq j<J,
\end{equation}
where $C=C(d,s,\gamma,\Omega,C_{\mathrm{geo}},C_\ast)$.
The same estimate holds with $A_j'$ replaced by $A_j^{(p)}$ for any
fixed $p\in\NN_+$, with $C_0$ and $C$ depending additionally on $p$.
\end{lemma}

The proof uses the following local energy estimate, valid for $0<s<1$
and also for balls intersecting $\partial\Omega$. The boundary case is treated in
Appendix~\ref{app:local-energy-error}.

\begin{theorem}[local energy estimate, {\cite[Theorem~5.3]{BorthagarayLeykekhmanNochetto2021}}]\label{thm:local-energy-error}
Let $0<s<1$, $u\in\widetilde H^s(\Omega)$ and $u_h:=R_h^s u$.
For any ball $B_R:=B_R(x_0)$ with $R\geq16h$ and any $v_h\in V_h$,
\begin{equation}\label{eqn:local-energy-error-split}
 \begin{aligned}
 &|u-u_h|_{H^s(B_{R/2})}^2
 \leq C|u-v_h|_{H^s(B_R)}^2
       +CR^{-2s}\|u-v_h\|_{L^2(\Omega\cap B_R)}^2\\
 &\quad+CR^{d+2s}\left(\int_{\Omega\cap B_R^c}
       \frac{|u-v_h|}{|x-x_0|^{d+2s}}\,\mathrm{d}x\right)^2\\
 &\quad+CR^{-2s}\|u-u_h\|_{L^2(\Omega\cap B_R)}^2
       +CR^{d+2s}\left(\int_{\Omega\cap B_R^c}
       \frac{|u-u_h|}{|x-x_0|^{d+2s}}\,\mathrm{d}x\right)^2.
 \end{aligned}
\end{equation}
Here $C$ depends only on $d$, $s$ and $\gamma$.
\end{theorem}

To prove Lemma~\ref{lem:annular-energy-error}, we first derive coupled local
energy and $L^2$ error estimates, then close them by weighted summation and
absorption. Until $C_\ast$ is fixed, $C$ may depend on $d$, $s$, $\gamma$,
$\Omega$ and $C_{\mathrm{geo}}$, but is independent of $h$, $J$ and $C_\ast$.

\subsection{Annular energy estimate}
Set $\rho:=z-z_h$. Fix $0\leq j<J$ and set $\tilde r_j:=r_j/64$.
Choose a covering
\[
 \overline{A_j'}\subset\bigcup_{i=1}^N B_{\tilde r_j}(x_i),
 \qquad x_i\in A_j',
\]
with $N$ and the overlap bounded independently of $j$.
Lemma~\ref{lem:frac-localize} (localization) gives
\begin{equation}\label{eqn:annular-localization}
 |\rho|_{H^s(A_j')}^2
 \leq\sum_{i=1}^N|\rho|_{H^s(B_{2\tilde r_j}(x_i))}^2
       +Cr_j^{-2s}\|\rho\|_{L^2(\Omega_j')}^2.
\end{equation}
Applying Theorem~\ref{thm:local-energy-error} (local energy estimate)
with $u=z$, $u_h=z_h$ and $v_h=I_hz$, we obtain
\begin{equation}\label{eqn:frac-local-error}
 \begin{aligned}
 |\rho|_{H^s(B_{2\tilde r_j}(x_i))} & 
 \leq C|z-I_hz|_{H^s(B_{4\tilde r_j}(x_i))}
       +Cr_j^{-s}\|z-I_hz\|_{L^2(\Omega\cap B_{4\tilde r_j}(x_i))}\\
 &\quad+C\mc T_{j,i}(z-I_hz)
       +Cr_j^{-s}\|\rho\|_{L^2(\Omega\cap B_{4\tilde r_j}(x_i))}
       +C\mc T_{j,i}(\rho),
 \end{aligned}
\end{equation}
where
\[
 \mc T_{j,i}(v):=r_j^{d/2+s}
 \int_{\Omega\cap B_{4\tilde r_j}^c(x_i)}
 \frac{|v(x)|}{|x-x_i|^{d+2s}}\,\mathrm{d}x.
\]
For $v\in L^2(\Omega)$, dyadic decomposition and Cauchy--Schwarz give
\begin{equation}\label{eqn:annular-tail}
 \begin{aligned}
 \mc T_{j,i}(v)
 &\leq Cr_j^{d/2+s}\biggl(
       \sum_{k=0}^{J-1}r_{j\wedge k}^{-d-2s}\|v\|_{L^1(\Omega_k)}
       +r_j^{-d-2s}\|v\|_{L^1(\Omega_\ast)}\biggr)\\
 &\leq Cr_j^{d/2+s}\biggl(
       \sum_{k=0}^{J-1}r_{j\wedge k}^{-d-2s}r_k^{d/2}
       \|v\|_{L^2(\Omega_k)}
       +r_j^{-d-2s}H^{d/2}\|v\|_{L^2(\Omega_\ast)}\biggr)\\
 &=C\sum_{k=0}^{J-1}2^{-(d/2+s)|j-k|}
       r_k^{-s}\|v\|_{L^2(\Omega_k)}
       +Cr_j^{-d/2-s}H^{d/2}\|v\|_{L^2(\Omega_\ast)}.
 \end{aligned}
\end{equation}
To justify the bounds in the first step, recall that
$r_j/4\leq|x_i-x_0|\leq2r_j$ and $4\tilde r_j=r_j/16$.
For $x\in\Omega_k\cap B_{4\tilde r_j}^c(x_i)$,
\[
 |x-x_i|\geq
 \begin{cases}
 r_k/2-2r_j\geq r_k/4,&k\leq j-3,\\
 r_j/4-r_k\geq r_j/8,&k\geq j+3,\\
 r_j/16\geq r_{j\wedge k}/64,&|k-j|\leq2.
 \end{cases}
\]
Thus $|x-x_i|\geq r_{j\wedge k}/64$ in all three cases.
On $\Omega_\ast\cap B_{4\tilde r_j}^c(x_i)$, the exclusion of the same
ball gives $|x-x_i|\geq r_j/16$, yielding the core term.

By \eqref{eqn:dual-core-local-interp}--\eqref{eqn:dual-core-local-interp-core}
and their versions on enlarged annuli,
\begin{align}
 h^{-s}\|z-I_hz\|_{L^2(\Omega_k')}
 +|z-I_hz|_{H^s(A_k')}
 &\leq Ch^{d/2+s}\ell_h^\xi r_k^{-d/2},
 \label{eqn:annular-interp-error}\\
 h^{-s}\|z-I_hz\|_{L^2(\Omega_\ast)}
 +|z-I_hz|_{H^s(A_\ast)}
 &\leq Ch^s\ell_h^\xi.
 \label{eqn:core-interp-error}
\end{align}
Substituting \eqref{eqn:annular-tail}--\eqref{eqn:core-interp-error}
into \eqref{eqn:frac-local-error} and using
$B_{4\tilde r_j}(x_i)\subset A_j''$ and $h\leq r_j$, we obtain
\begin{equation}\label{eqn:frac-local-error-ball}
 \begin{aligned}
 &|\rho|_{H^s(B_{2\tilde r_j}(x_i))}
 \leq Ch^{d/2+s}\ell_h^\xi r_j^{-d/2}
       +Ch^{d/2+2s}\ell_h^\xi
        \sum_{k=0}^{J-1}2^{-(d/2+s)|j-k|}r_k^{-d/2-s}\\
 &\quad+Cr_j^{-d/2-s}H^{d/2}h^{2s}\ell_h^\xi
       +Cr_j^{-s}\|\rho\|_{L^2(\Omega_j'')}\\
 &\quad+C\sum_{k=0}^{J-1}2^{-(d/2+s)|j-k|}
        r_k^{-s}\|\rho\|_{L^2(\Omega_k)}
       +Cr_j^{-d/2-s}H^{d/2}\|\rho\|_{L^2(\Omega_\ast)}\\
 &\leq CC_\ast^{d/2-s}h^{d/2+s}r_j^{-d/2}
        \ell_h^{\beta(d/2-s)+2\xi}
       +C\sum_{k=0}^{J-1}2^{-(d/2+s)|j-k|}
        r_k^{-s}\|\rho\|_{L^2(\Omega_k)}.
 \end{aligned}
\end{equation}
To justify the last inequality, we estimate the terms as follows.
\begin{enumerate}
\renewcommand{\labelenumi}{\textup{(\roman{enumi})}}
\item The local interpolation term
$Ch^{d/2+s}\ell_h^\xi r_j^{-d/2}$ is bounded by the explicit term in
the last line, since $d/2-s\geq0$, $C_\ast\geq1$ and $\ell_h\geq1$.

\item For the annular interpolation tail, using $h\leq r_k$ and
summing a geometric series gives
\[
 \begin{aligned}
 &h^{d/2+2s}\ell_h^\xi
   \sum_{k=0}^{J-1}2^{-(d/2+s)|j-k|}r_k^{-d/2-s}\\
 &\quad\leq h^{d/2+s}\ell_h^\xi
   \sum_{k=0}^{J-1}2^{-(d/2+s)|j-k|}r_k^{-d/2}
   \leq Ch^{d/2+s}\ell_h^\xi r_j^{-d/2}.
 \end{aligned}
\]

\item For the core interpolation tail, $r_j\geq H$ and
\eqref{eqn:terminal-scale} give
\[
 \begin{aligned}
 r_j^{-d/2-s}H^{d/2}h^{2s}\ell_h^\xi
 &\leq r_j^{-d/2}H^{d/2-s}h^{2s}\ell_h^\xi \\
 &=C_\ast^{d/2-s}h^{d/2+s}r_j^{-d/2}
   \ell_h^{\beta(d/2-s)+\xi}.
 \end{aligned}
\]

\item For the local $L^2$ error, use $r_k\simeq r_j$ when
$|k-j|\leq2$. Moreover, $\Omega_j''$ meets $\Omega_\ast$ only if
$j\geq J-2$, in which case $r_j\simeq H$. Thus
\[
 \begin{aligned}
 r_j^{-s}\|\rho\|_{L^2(\Omega_j'')}
 &\leq C\sum_{\substack{0\leq k<J\\|k-j|\leq2}}
       2^{-(d/2+s)|j-k|}r_k^{-s}\|\rho\|_{L^2(\Omega_k)}\\
 &\quad+Cr_j^{-d/2-s}H^{d/2}\|\rho\|_{L^2(\Omega_\ast)}.
 \end{aligned}
\]
These contributions are covered by the fifth and sixth terms.

\setcounter{enumi}{5}
\item For the core error tail, we use an Aubin--Nitsche argument.
By Theorem~\ref{thm:sobolev-shift-lip}
(global Sobolev regularity) and Lemma~\ref{lem:global-fractional}
(global fractional-order interpolation estimates), energy best approximation gives
$|\rho|_{H^s(\RR^d)}\leq Ch^s\ell_h^\xi$.
Let $w\in\widetilde H^s(\Omega)$ solve
$a_s(w,v)=(\rho,v)_\Omega$ for all $v\in\widetilde H^s(\Omega)$.
Applying the same regularity and interpolation estimates to $w$ and
using Galerkin orthogonality, we obtain
\[
 \begin{aligned}
 \|\rho\|_{L^2(\Omega)}^2
 &=a_s(\rho,w-I_hw)\\
 &\leq Ch^s\ell_h^\xi|\rho|_{H^s(\RR^d)}\|\rho\|_{L^2(\Omega)}
 \leq Ch^{2s}\ell_h^{2\xi}\|\rho\|_{L^2(\Omega)}.
 \end{aligned}
\]
Thus we have the global error bounds
\begin{equation}\label{eqn:global-error-bounds}
 |\rho|_{H^s(\RR^d)}\leq Ch^s\ell_h^\xi,
 \quad
 \|\rho\|_{L^2(\Omega)}\leq Ch^{2s}\ell_h^{2\xi}.
\end{equation}
Since $r_j\geq H$, \eqref{eqn:terminal-scale} gives
\[
 \begin{aligned}
 r_j^{-d/2-s}H^{d/2}\|\rho\|_{L^2(\Omega_\ast)}
 &\leq Cr_j^{-d/2-s}H^{d/2}h^{2s}\ell_h^{2\xi}\\
 &\leq Cr_j^{-d/2}H^{d/2-s}h^{2s}\ell_h^{2\xi}\\
 &=CC_\ast^{d/2-s}h^{d/2+s}r_j^{-d/2}
        \ell_h^{\beta(d/2-s)+2\xi}.
 \end{aligned}
\]
\end{enumerate}

Combining \eqref{eqn:annular-localization} and
\eqref{eqn:frac-local-error-ball} with \eqref{eqn:scale-energy-def} yields
\begin{equation}\label{eqn:local-energy-error}
 \begin{aligned}
 \mc E_j(\rho)
 &\leq CC_\ast^{d/2-s}h^{d/2+s}r_j^{-d/2}
       \ell_h^{\beta(d/2-s)+2\xi} \\
  &\quad +C\sum_{k=0}^{J-1}2^{-(d/2+s)|j-k|}
       r_k^{-s}\|\rho\|_{L^2(\Omega_k)}.
 \end{aligned}
\end{equation}

\subsection{Annular \texorpdfstring{$L^2$}{L2} estimates by duality}
Fix $0\leq j<J$. For $v_j\in C_0^\infty(\Omega_j)$ with
$\|v_j\|_{L^2(\Omega_j)}=1$, let $w_j$ solve
\eqref{eqn:local-source-annuli} and set $\zeta_j:=w_j-I_hw_j$.
Galerkin orthogonality gives
\[
 (\rho,v_j)=a_s(\rho,w_j)=a_s(\rho,\zeta_j).
\]
Since both functions vanish in $\Omega^c$, we split the bilinear form as
\begin{equation}\label{eqn:integral-split}
 \begin{aligned}
 a_s(\rho,\zeta_j)
 &=\frac{c_{d,s}}{2}\int_\Omega\int_\Omega
   \frac{(\rho(x)-\rho(y))(\zeta_j(x)-\zeta_j(y))}{|x-y|^{d+2s}}
   \,\mathrm{d}y\mathrm{d}x\\
 &\quad+c_{d,s}\int_\Omega\int_{\Omega^c}
   \frac{\rho(x)\zeta_j(x)}{|x-y|^{d+2s}}\,\mathrm{d}y\mathrm{d}x.
 \end{aligned}
\end{equation}
We use the local interpolation estimates \eqref{eqn:dual-local-interp}
and \eqref{eqn:dual-local-interp-core}, also on fixed enlarged regions.

\textbf{Step 1: cross term.}
Lemma~\ref{lem:delta-upper-bound} (exterior kernel bound) gives
\[
 \begin{aligned}
 &\left|\int_\Omega\int_{\Omega^c}
   \frac{\rho(x)\zeta_j(x)}{|x-y|^{d+2s}}\,\mathrm{d}y\mathrm{d}x\right|
 \leq C\int_\Omega|\rho\zeta_j|\delta^{-2s}\,\mathrm{d}x\\
 &\quad\leq C\sum_{k=0}^{J-1}
   \|\rho\delta^{-s}\|_{L^2(\Omega_k)}
   \|\zeta_j\delta^{-s}\|_{L^2(\Omega_k)}
   +C\|\rho\delta^{-s}\|_{L^2(\Omega_\ast)}
   \|\zeta_j\delta^{-s}\|_{L^2(\Omega_\ast)}.
 \end{aligned}
\]
By \eqref{eqn:delta-upper-bound} and the zero extension of $\rho$,
\[
 \begin{aligned}
 \|\rho\delta^{-s}\|_{L^2(\Omega_k)}^2
 &\leq C\int_{\Omega_k}\int_{\Omega^c}
   \frac{|\rho(x)|^2}{|x-y|^{d+2s}}\,\mathrm{d}y\mathrm{d}x\\
 &= C\int_{\Omega_k}\int_{\Omega^c\cap A_k'}
   \frac{|\rho(x)-\rho(y)|^2}{|x-y|^{d+2s}}\,\mathrm{d}y\mathrm{d}x
   +C\int_{\Omega_k}\int_{\Omega^c\setminus A_k'}
   \frac{|\rho(x)|^2}{|x-y|^{d+2s}}\,\mathrm{d}y\mathrm{d}x\\
 &\leq C\bigl(|\rho|_{H^s(A_k')}^2
   +r_k^{-2s}\|\rho\|_{L^2(\Omega_k')}^2\bigr)
   \leq C\mc E_k(\rho)^2.
 \end{aligned}
\]
On the core, \eqref{eqn:delta-upper-bound} and \eqref{eqn:global-error-bounds} give $\|\rho\delta^{-s}\|_{L^2(\Omega_\ast)}
 \leq C|\rho|_{H^s(\RR^d)}
 \leq Ch^s\ell_h^\xi$.
Since $r_J\simeq H$ and
$\kappa_{jJ}\simeq(H/r_j)^{d/2-s}$, the local interpolation estimates
\eqref{eqn:dual-local-interp}--\eqref{eqn:dual-local-interp-core} give
\begin{equation}\label{eqn:cross-term}
 \begin{aligned}
 &\left|\int_\Omega\int_{\Omega^c}
   \frac{\rho(x)\zeta_j(x)}{|x-y|^{d+2s}}\,\mathrm{d}y\mathrm{d}x\right| 
   \leq Ch^s\ell_h^\xi\sum_{k=0}^{J-1}\kappa_{jk}\mc E_k(\rho)
   +Ch^{2s}\ell_h^{2\xi}\left(\frac{H}{r_j}\right)^{d/2-s}.
 \end{aligned}
\end{equation}

\textbf{Step 2: main term.}
For the first term in \eqref{eqn:integral-split}, we separate neighboring
regions from separated annuli and the core:
\begin{equation}\label{eqn:main-term-split}
 \begin{aligned}
 &\frac{c_{d,s}}{2}\int_\Omega\int_\Omega
   \frac{|\rho(x)-\rho(y)||\zeta_j(x)-\zeta_j(y)|}{|x-y|^{d+2s}}
   \,\mathrm{d}y\mathrm{d}x\\
 &\quad\leq\frac{c_{d,s}}{2}\Biggl[
   \sum_{k=0}^{J-1}\int_{\Omega_k'}\int_{\Omega_k'}
      +\int_{\Omega_\ast'}\int_{\Omega_\ast'}
   +2\sum_{\substack{0\leq i<k<J\\k-i\geq2}}
   \int_{\Omega_i}\int_{\Omega_k}
   +2\sum_{k=0}^{J-2}\int_{\Omega_k}\int_{\Omega_\ast}\Biggr]\\
 &\qquad\quad\times
   \frac{|\rho(x)-\rho(y)||\zeta_j(x)-\zeta_j(y)|}{|x-y|^{d+2s}}
   \,\mathrm{d}y\mathrm{d}x.
 \end{aligned}
\end{equation}

\fbox{Step 2-1: near field.}
By Cauchy--Schwarz and the local interpolation estimates
\eqref{eqn:dual-local-interp}--\eqref{eqn:dual-local-interp-core},
the first two terms in \eqref{eqn:main-term-split} are bounded by
\begin{equation}\label{eqn:duality-near-field}
 \begin{aligned}
 &\sum_{k=0}^{J-1}|\rho|_{H^s(A_k')}|\zeta_j|_{H^s(A_k')}
   +|\rho|_{H^s(A_\ast')}|\zeta_j|_{H^s(A_\ast')} \\
 &\quad\leq Ch^s\ell_h^\xi\sum_{k=0}^{J-1}\kappa_{jk}\mc E_k(\rho)
   +Ch^{2s}\ell_h^{2\xi}\left(\frac{H}{r_j}\right)^{d/2-s}.
 \end{aligned}
\end{equation}

\fbox{Step 2-2: far field.}
On $\Omega_i\times\Omega_k$ with $0\leq i<k<J$ and $k-i\geq2$,
we have $|x-y|\simeq r_i$. Expanding the product gives
\[
 \int_{\Omega_i}\int_{\Omega_k}
 \frac{|\rho(x)-\rho(y)||\zeta_j(x)-\zeta_j(y)|}{|x-y|^{d+2s}}
 \,\mathrm{d}y\mathrm{d}x
 \leq Cr_i^{-d-2s}(I_1+I_2+I_3+I_4),
\]
where
\[
 \begin{aligned}
 I_1&:=|\Omega_k|\|\rho\zeta_j\|_{L^1(\Omega_i)},
 &I_4&:=|\Omega_i|\|\rho\zeta_j\|_{L^1(\Omega_k)},\\
 I_2&:=\|\zeta_j\|_{L^1(\Omega_i)}\|\rho\|_{L^1(\Omega_k)},
 &I_3&:=\|\rho\|_{L^1(\Omega_i)}\|\zeta_j\|_{L^1(\Omega_k)}.
 \end{aligned}
\]
By \eqref{eqn:scale-energy-def} and the interpolation estimate
\eqref{eqn:dual-local-interp}, with $\|v_j\|_{L^2(\Omega_j)}=1$,
we have, for $0\leq m<J$, $\|\rho\|_{L^2(\Omega_m)}\leq r_m^s\mc E_m(\rho)$ and $\|\zeta_j\|_{L^2(\Omega_m)}
 \leq Ch^{2s}\ell_h^\xi\kappa_{jm}$.
For $I_1$ and $I_4$, Cauchy--Schwarz gives
\[
 I_1\leq Cr_k^d h^{2s}\ell_h^\xi\kappa_{ji}r_i^s\mc E_i(\rho),
 \quad
 I_4\leq Cr_i^d h^{2s}\ell_h^\xi\kappa_{jk}r_k^s\mc E_k(\rho).
\]
Using $\sum_{k\geq i+2}r_k^d\leq Cr_i^d$,
$\sum_{i=0}^{k-2}r_i^{-2s}\leq Cr_k^{-2s}$ and $h\leq r_k$, we then obtain
\[
 \begin{aligned}
 \sum_{\substack{0\leq i<k<J\\k-i\geq2}}r_i^{-d-2s}(I_1+I_4)
 &\leq Ch^{2s}\ell_h^\xi\sum_{k=0}^{J-1}r_k^{-s}\kappa_{jk}\mc E_k(\rho)
 \leq Ch^s\ell_h^\xi\sum_{k=0}^{J-1}\kappa_{jk}\mc E_k(\rho).
 \end{aligned}
\]

For the mixed terms, Cauchy--Schwarz gives
\[
 \begin{aligned}
 I_2&\leq Ch^{2s}\ell_h^\xi
       r_i^{d/2}\kappa_{ji}r_k^{d/2+s}\mc E_k(\rho),\\
 I_3&\leq Ch^{2s}\ell_h^\xi
       r_i^{d/2+s}\mc E_i(\rho)r_k^{d/2}\kappa_{jk}.
 \end{aligned}
\]
We use the dyadic bounds
\begin{equation}\label{eqn:kappa-inequ}
 \begin{aligned}
 \sum_{i=0}^{k-2}r_i^{-d/2-2s}\kappa_{ji}
 &\leq Cr_k^{-d/2-2s}\kappa_{jk},\qquad 0\leq k\leq J,\\
 \sum_{k=i+2}^{J-1}r_k^{d/2+s}\kappa_{jk}
 &\leq Cr_i^{d/2+s}\kappa_{ji},\qquad 0\leq i<J.
 \end{aligned}
\end{equation}
Indeed, by \eqref{eqn:kappa}, splitting the first sum at $i=j$ gives
\[
 \sum_{i=0}^{k-2}r_i^{-d/2-2s}\kappa_{ji}
 \leq
 \begin{cases}
 Cr_j^{d/2}r_k^{-d-2s}&k\leq j+1,\\
 C\bigl(r_j^{-d/2-2s}+r_j^{s-d/2}r_k^{-3s}\bigr)&k\geq j+2,
 \end{cases}
 \leq Cr_k^{-d/2-2s}\kappa_{jk}.
\]
For the second sum, the summands are comparable to $r_j^{d/2}r_k^s$
for $k\leq j$ and $r_j^{s-d/2}r_k^d$ for $k>j$; summing the two
geometric series proves the bound. Here the factors on neighboring
annuli are comparable to those in \eqref{eqn:kappa}.
Thus \eqref{eqn:kappa-inequ} yields
\[
 \begin{aligned}
 \sum_{\substack{0\leq i<k<J\\k-i\geq2}}r_i^{-d-2s}I_2
 &\leq Ch^{2s}\ell_h^\xi\sum_{k=0}^{J-1}r_k^{d/2+s}\mc E_k(\rho)
       \sum_{i=0}^{k-2}r_i^{-d/2-2s}\kappa_{ji}\\
 &\leq Ch^s\ell_h^\xi\sum_{k=0}^{J-1}\kappa_{jk}\mc E_k(\rho),\\
 \sum_{\substack{0\leq i<k<J\\k-i\geq2}}r_i^{-d-2s}I_3
 &\leq Ch^s\ell_h^\xi\sum_{i=0}^{J-1}r_i^{-d/2-s}\mc E_i(\rho)
       \sum_{k=i+2}^{J-1}r_k^{d/2+s}\kappa_{jk}\\
 &\leq Ch^s\ell_h^\xi\sum_{i=0}^{J-1}\kappa_{ji}\mc E_i(\rho),
 \end{aligned}
\]
where we also used $h^s\leq r_k^s$.
Combining the above estimates for $I_1,\ldots,I_4$ gives
\begin{equation}\label{eqn:duality-annular-far-field}
 \begin{aligned}
 &\sum_{\substack{0\leq i<k<J\\k-i\geq2}}
   \int_{\Omega_i}\int_{\Omega_k}
   \frac{|\rho(x)-\rho(y)||\zeta_j(x)-\zeta_j(y)|}{|x-y|^{d+2s}}
   \,\mathrm{d}y\mathrm{d}x \leq Ch^s\ell_h^\xi
   \sum_{k=0}^{J-1}\kappa_{jk}\mc E_k(\rho).
 \end{aligned}
\end{equation}

\fbox{Step 2-3: interaction with the core.}
For the last term in \eqref{eqn:main-term-split}, we have
$|x-y|\simeq r_k$ on $\Omega_k\times\Omega_\ast$ for $k\leq J-2$.
Define $I_1,\ldots,I_4$ as above with $(\Omega_i,\Omega_k)$ replaced by
$(\Omega_k,\Omega_\ast)$. The same expansion gives
\[
 \int_{\Omega_k}\int_{\Omega_\ast}
 \frac{|\rho(x)-\rho(y)||\zeta_j(x)-\zeta_j(y)|}{|x-y|^{d+2s}}
 \,\mathrm{d}y\mathrm{d}x
 \leq Cr_k^{-d-2s}(I_1+I_2+I_3+I_4).
\]
The terms involving $\rho$ on $\Omega_k$ satisfy
\[
 \begin{aligned}
 r_k^{-d-2s}I_1
 &\leq CH^d r_k^{-d-s}h^{2s}\ell_h^\xi\kappa_{jk}\mc E_k(\rho),\\
 r_k^{-d-2s}I_3
 &\leq CH^{d-s}r_k^{-d/2-s}h^{2s}\ell_h^\xi
       r_j^{s-d/2}\mc E_k(\rho).
 \end{aligned}
\]
Since $H\leq r_j,r_k$, \eqref{eqn:kappa} gives
$H^{d-s}r_k^{-d/2-s}r_j^{s-d/2}\leq CH^{-s}\kappa_{jk}$.
Together with $h\leq H$, this implies
\[
 \sum_{k=0}^{J-2}r_k^{-d-2s}(I_1+I_3)
 \leq Ch^s\ell_h^\xi\sum_{k=0}^{J-1}\kappa_{jk}\mc E_k(\rho).
\]
For the terms involving $\rho$ on the core, the global $L^2$ error
bound and \eqref{eqn:dual-local-interp-core} give
\[
 \begin{aligned}
 r_k^{-d-2s}I_2
 &\leq CH^{d/2}r_k^{-d/2-2s}h^{4s}\ell_h^{3\xi}\kappa_{jk},\\
 r_k^{-d-2s}I_4
 &\leq Cr_k^{-2s}h^{4s}\ell_h^{3\xi}
       \left(\frac{H}{r_j}\right)^{d/2-s}.
 \end{aligned}
\]
Applying \eqref{eqn:kappa-inequ} with $k=J$ and using $r_J\simeq H$, we obtain
\[
 \sum_{k=0}^{J-2}r_k^{-d-2s}(I_2+I_4)
 \leq CH^{-2s}h^{4s}\ell_h^{3\xi}
       \left(\frac{H}{r_j}\right)^{d/2-s}.
\]

Combining the above estimates for $I_1,\ldots,I_4$ gives
\begin{equation}\label{eqn:duality-core-far-field}
 \begin{aligned}
 &\sum_{k=0}^{J-2}\int_{\Omega_k}\int_{\Omega_\ast}
   \frac{|\rho(x)-\rho(y)||\zeta_j(x)-\zeta_j(y)|}{|x-y|^{d+2s}}
   \,\mathrm{d}y\mathrm{d}x\\
 &\quad\leq Ch^s\ell_h^\xi
   \sum_{k=0}^{J-1}\kappa_{jk}\mc E_k(\rho)
   +CH^{-2s}h^{4s}\ell_h^{3\xi}
    \left(\frac{H}{r_j}\right)^{d/2-s}.
 \end{aligned}
\end{equation}

Combining \eqref{eqn:duality-near-field},
\eqref{eqn:duality-annular-far-field} and
\eqref{eqn:duality-core-far-field} with \eqref{eqn:cross-term},
using $h\leq H$ and taking the supremum over normalized $v_j$ yields
\begin{equation}\label{eqn:local-L2-error}
 \begin{aligned}
 \|\rho\|_{L^2(\Omega_j)}
 &\leq Ch^s\ell_h^\xi\sum_{k=0}^{J-1}\kappa_{jk}\mc E_k(\rho)\\
 &\quad+Cr_j^{s-d/2}h^{d/2+s}
   \bigl(C_\ast^{d/2-2s}\ell_h^{3\xi+\beta(d/2-2s)}
        +C_\ast^{d/2-s}\ell_h^{2\xi+\beta(d/2-s)}\bigr),
 \end{aligned}
\end{equation}
where we substituted $H=C_\ast h\ell_h^\beta$ from \eqref{eqn:terminal-scale}.

\subsection{Closing the coupled annular estimates}
The scaled energy and local $L^2$ error control each other through the coupled
estimates derived above.
Recall from \eqref{eqn:local-energy-error} and
\eqref{eqn:local-L2-error} that, for $0\leq j<J$,
\[
 \begin{aligned}
 \mc E_j(\rho)
 &\leq CC_\ast^{d/2-s}h^{d/2+s}r_j^{-d/2}
       \ell_h^{\beta(d/2-s)+2\xi} +C\sum_{k=0}^{J-1}2^{-(d/2+s)|j-k|}
       r_k^{-s}\|\rho\|_{L^2(\Omega_k)},\\
 \|\rho\|_{L^2(\Omega_j)}
 &\leq Ch^s\ell_h^\xi\sum_{k=0}^{J-1}\kappa_{jk}\mc E_k(\rho)\\
 &\quad+Cr_j^{s-d/2}h^{d/2+s}
   \bigl(C_\ast^{d/2-2s}\ell_h^{3\xi+\beta(d/2-2s)}
        +C_\ast^{d/2-s}\ell_h^{2\xi+\beta(d/2-s)}\bigr).
 \end{aligned}
\]
Here $H=C_\ast h\ell_h^\beta$, with $\beta\geq0$ and $C_\ast\geq1$
to be chosen, and $C$ is independent of $h$, $J$ and $C_\ast$.
Set
\[
 M:=\max_{0\leq k<J}r_k^{d/2-s}\|\rho\|_{L^2(\Omega_k)}.
\]
Substituting $\|\rho\|_{L^2(\Omega_k)}\leq Mr_k^{s-d/2}$ into
\eqref{eqn:local-energy-error} and summing the geometric series gives
\begin{equation}\label{eqn:energy-in-M}
 \mc E_j(\rho)
 \leq Cr_j^{-d/2}\bigl(M+C_\ast^{d/2-s}h^{d/2+s}
                   \ell_h^{\beta(d/2-s)+2\xi}\bigr).
\end{equation}
Moreover, \eqref{eqn:kappa} and $r_J\simeq H$ imply
\[
 \begin{aligned}
 \sum_{k=0}^{J-1}\kappa_{jk}r_k^{-d/2}
 &\leq Cr_j^{d/2}\sum_{k=0}^{j}r_k^{-d}
       +Cr_j^{s-d/2}\sum_{k=j+1}^{J-1}r_k^{-s}\\
 &\leq Cr_j^{-d/2}+Cr_j^{s-d/2}H^{-s}
 \leq Cr_j^{s-d/2}H^{-s}.
 \end{aligned}
\]
Inserting \eqref{eqn:energy-in-M} into \eqref{eqn:local-L2-error},
multiplying by $r_j^{d/2-s}$ and taking the maximum over $j$, we obtain
\[
 \begin{aligned}
 M&\leq C\left(\frac{h}{H}\right)^s\ell_h^\xi
   \bigl(M+C_\ast^{d/2-s}h^{d/2+s}\ell_h^{\beta(d/2-s)+2\xi}\bigr)\\
 &\quad+Ch^{d/2+s}
   \bigl(C_\ast^{d/2-2s}\ell_h^{3\xi+\beta(d/2-2s)}
        +C_\ast^{d/2-s}\ell_h^{2\xi+\beta(d/2-s)}\bigr).
 \end{aligned}
\]
Choose $\beta:=\xi/s$ and then $C_\ast$ sufficiently large that
$C(h/H)^s\ell_h^\xi=CC_\ast^{-s}\leq1/2$.
Absorbing the term containing $M$ gives
\[
 M\leq CC_\ast^{d/2-s}h^{d/2+s}\ell_h^{d\xi/(2s)+\xi}.
\]
Returning to \eqref{eqn:energy-in-M}, we conclude that
\[
 \mc E_j(\rho)
 \leq CC_\ast^{d/2-s}r_j^{-d/2}h^{d/2+s}\ell_h^{d\xi/(2s)+\xi},
 \qquad 0\leq j<J.
\]
This proves \eqref{eqn:annular-energy-error}. The same argument applies
to any fixed enlargement of the annuli.

\section{Proof of almost \texorpdfstring{$L^\infty$}{L-infinity}-stability}\label{sec:linf-stability}
\begin{proof}[Proof of Theorem~\ref{thm:linf-stab}]
Set $u_h:=R_h^s u$. For $h\geq h_0>0$, the fractional Poincar\'e
and inverse inequalities give
\[
 \begin{aligned}
 \|u_h\|_{L^2(\Omega_h)}^2
 &\leq Ca_s(u,u_h)
 \leq Ch^{-2s}\|u\|_{L^\infty(\Omega)}\|u_h\|_{L^2(\Omega_h)},\\
 \|u_h\|_{L^\infty(\Omega_h)}
 &\leq Ch^{-d/2}\|u_h\|_{L^2(\Omega_h)}
 \leq C h_0^{-d/2 -2s}\|u\|_{L^\infty(\Omega)},
 \end{aligned}
\]
where we used $2s\leq1$. It therefore suffices to prove
\eqref{eqn:linf-stab} for sufficiently small $h$.
Recall that $\alpha$ and $\xi$ are given in \eqref{eqn:xi-alpha}, so that
$\alpha-1=d\xi/(2s)+\xi$.

\textbf{Step 1: duality.}
Choose $x_0\in T_0\in\mc T_h$ such that
$|u_h(x_0)|=\|u_h\|_{L^\infty(\Omega)}$. The inverse inequality and
$L^2$ duality give
\[
 \|u_h\|_{L^\infty(\Omega)}
 \leq Ch^{-d/2}\|u_h\|_{L^2(T_0)}
 =Ch^{-d/2}\sup_{\substack{\varphi\in C_0^\infty(T_0)\\
                 \|\varphi\|_{L^2(T_0)}=1}}|(u_h,\varphi)_\Omega|.
\]
For each such $\varphi$, let $z$ solve the element-source dual problem
introduced in \eqref{eqn:local-source},
\begin{equation}\label{eqn:dual-problem}
 \begin{cases}
 (-\Delta)^s z=\varphi &\text{in }\Omega,\\
 z=0 &\text{in }\Omega^c,
 \end{cases}
\end{equation}
and set $z_h:=R_h^s z$. Galerkin orthogonality yields
\[
 \begin{aligned}
 |(u_h,\varphi)_\Omega|
 &=|a_s(u_h,z)|=|a_s(u_h,z_h)|=|a_s(u,z_h)|\\
 &=|(u,(-\Delta)^s z_h)_\Omega|
 \leq\|u\|_{L^\infty(\Omega)}\|(-\Delta)^s z_h\|_{L^1(\Omega)}.
 \end{aligned}
\]
The last pairing is well defined: since $2s\leq1$, zero extension implies
$V_h\subset H^{2s}(\RR^d)$ and hence $(-\Delta)^s z_h\in L^2(\RR^d)$.
It remains to prove, uniformly over the normalized sources $\varphi$, that
\begin{equation}\label{eqn:discrete-dual-L1}
 \|(-\Delta)^s z_h\|_{L^1(\Omega)}\leq Ch^{d/2}\ell_h^\alpha.
\end{equation}

\textbf{Step 2: dyadic localization.}
Use the decomposition from Section~\ref{sec:annular-error}, with
$H=C_\ast h\ell_h^{\xi/s}$ and $C_\ast$ fixed as in
Lemma~\ref{lem:annular-energy-error} (annular energy error).
Choose smooth cutoffs $0\leq\psi_j,\psi_\ast\leq1$, supported in $A_j'$ and
$A_\ast$, respectively, such that
\[
 \begin{gathered}
 \sum_{j=0}^{J-1}\psi_j+\psi_\ast=1 \text{ on }\Omega, \quad
 \dist(\supp\psi_j,(A_j')^c)\geq r_j/C,\\
 |D^\nu\psi_j|\leq Cr_j^{-|\nu|},\quad
 |D^\nu\psi_\ast|\leq CH^{-|\nu|},\quad |\nu|\leq2.
 \end{gathered}
\]
Set $z_{h,j}:=\psi_jz_h$ and $z_{h,\ast}:=\psi_\ast z_h$, so that
$z_h=\sum_{j=0}^{J-1}z_{h,j}+z_{h,\ast}$. Then
\[
 \|(-\Delta)^s z_h\|_{L^1(\Omega)}
 \leq\sum_{j=0}^{J-1}\|(-\Delta)^s z_{h,j}\|_{L^1(\Omega)}
      +\|(-\Delta)^s z_{h,\ast}\|_{L^1(\Omega)}.
\]
Since $\supp z_{h,j}\subset A_j'$ and
$\dist(A_j',(A_j'')^c)\geq r_j/C$, Cauchy--Schwarz and integration of
the exterior kernel give
\[
 \begin{aligned}
 \|(-\Delta)^s z_{h,j}\|_{L^1(\Omega)}
 &\leq\|(-\Delta)^s z_{h,j}\|_{L^1(A_j'')}
       +\|(-\Delta)^s z_{h,j}\|_{L^1(\Omega\setminus A_j'')}\\
 &\leq Cr_j^{d/2}|z_{h,j}|_{H^{2s}(\RR^d)}
       +C\int_{\Omega\setminus A_j''}\int_{A_j'}
         \frac{|z_{h,j}(y)|}{|x-y|^{d+2s}}\,\mathrm{d}y\mathrm{d}x\\
 &\leq Cr_j^{d/2}|z_{h,j}|_{H^{2s}(\RR^d)}
       +Cr_j^{-2s}\|z_{h,j}\|_{L^1(A_j')}\\
 &\leq Cr_j^{d/2}\bigl(|z_{h,j}|_{H^{2s}(\RR^d)}
       +r_j^{-2s}\|z_{h,j}\|_{L^2(A_j')}\bigr).
 \end{aligned}
\]
Lemma~\ref{lem:frac-cutoff} (cutoff estimate), with $2s\leq1$, gives
\[
 |z_{h,j}|_{H^{2s}(\RR^d)}
 \leq C\bigl(|z_h|_{H^{2s}(A_j')}
       +r_j^{-2s}\|z_h\|_{L^2(A_j')}\bigr).
\]
Applying the same argument on the core and using $r_J\simeq H$, we obtain
\begin{equation}\label{eqn:discrete-dual-localization}
 \begin{aligned}
 \|(-\Delta)^s z_h\|_{L^1(\Omega)}
 &\leq C\sum_{j=0}^{J-1}r_j^{d/2}
       \bigl(|z_h|_{H^{2s}(A_j')}+r_j^{-2s}\|z_h\|_{L^2(A_j')}\bigr)\\
 &\quad+CH^{d/2}
       \bigl(|z_h|_{H^{2s}(A_\ast')}+H^{-2s}\|z_h\|_{L^2(A_\ast')}\bigr).
 \end{aligned}
\end{equation}
We estimate the $L^2$ and $H^{2s}$ terms separately.

\textbf{Step 3: the $L^2$ terms.}
Suppose first that $2s<d$. Lemma~\ref{lem:local-L2-core}
(element sources: $L^2$ estimates) and
Lemma~\ref{lem:annular-energy-error} give
\[
 \begin{aligned}
 \|z_h\|_{L^2(A_j')}
 &\leq\|z-z_h\|_{L^2(A_j')}+\|z\|_{L^2(A_j')}\\
 &\leq Cr_j^{s-d/2}h^{d/2+s}\ell_h^{\alpha-1}
       +Cr_j^{2s-d/2}h^{d/2}.
 \end{aligned}
\]
On the core, $|x-y|\leq3H$ for $x\in\Omega_\ast'$ and $y\in T_0$.
Theorem~\ref{thm:Green-estimate} (Green's function estimate) and
Young's inequality imply
\[
 \|z\|_{L^2(\Omega_\ast')}
 \leq C\bigl\||\cdot|^{2s-d}\chi_{B_{3H}(0)}\bigr\|_{L^1(\RR^d)}
       \|\varphi\|_{L^2(T_0)}
 \leq CH^{2s}.
\]
Thus \eqref{eqn:global-error-bounds} and $H=C_\ast h\ell_h^{\xi/s}$ imply
\[
 \|z_h\|_{L^2(A_\ast')}
 \leq Ch^{2s}\ell_h^{2\xi}+CH^{2s}\leq CH^{2s}.
\]
Using $\sum_{j=0}^{J-1}r_j^{-s}\leq CH^{-s}$ and
$J+1\leq C\ell_h$, we conclude that
\begin{equation}\label{eqn:L2-sum}
 \begin{aligned}
 &\sum_{j=0}^{J-1}r_j^{d/2-2s}\|z_h\|_{L^2(A_j')}
       +H^{d/2-2s}\|z_h\|_{L^2(A_\ast')}\\
 &\quad\leq Ch^{d/2}\bigl((h/H)^s\ell_h^{\alpha-1}+\ell_h\bigr)+CH^{d/2}\\
 &\quad\leq Ch^{d/2}\ell_h^{\max\{d\xi/(2s),1\}}.
 \end{aligned}
\end{equation}

For $(s,d)=(1/2,1)$, the same argument with the critical bounds in
Lemma~\ref{lem:local-L2-core} and \eqref{eqn:Green-estimate} gives
\[
 \begin{aligned}
 \|z_h\|_{L^2(A_j')}
 &\leq Ch\ell_h^2+C\sqrt{hr_j}\bigl(1+|\log r_j|\bigr),\\
 \|z_h\|_{L^2(A_\ast')}
 &\leq Ch\ell_h^2+CH\bigl(1+|\log H|\bigr).
 \end{aligned}
\]
Using $H=C_\ast h\ell_h^2$ and $1+|\log r_j|\leq C\ell_h$ for $j<J$,
the same summation yields
\begin{equation}\label{eqn:L2-sum-critical}
 \begin{aligned}
 &\sum_{j=0}^{J-1}r_j^{-1/2}\|z_h\|_{L^2(A_j')}
       +H^{-1/2}\|z_h\|_{L^2(A_\ast')}\\
 &\quad\leq Ch^{1/2}\bigl((h/H)^{1/2}\ell_h^2+\ell_h^2\bigr)
 \leq Ch^{1/2}\ell_h^2.
 \end{aligned}
\end{equation}

\textbf{Step 4: the $H^{2s}$ terms.}
Take $\eps:=\min\{s/2,\ell_h^{-1}\}$, so that
$h^{-\eps}\eps^{-\xi}\leq C\ell_h^\xi$.
Write $z_h=(z_h-I_hz)+I_hz$. Applying local inverse inequalities to the
two finite element terms and then Corollary~\ref{cor:local-interp-estimate}
(annular and core interpolation estimates) on enlarged annuli gives
\[
 \begin{aligned}
 |z_h|_{H^{2s}(A_j')}
 &\leq Ch^{-s}|z_h-I_hz|_{H^s(A_j'')}
       +Ch^{-\eps}|z|_{H^{2s-\eps}(A_j^{(3)})}\\
 &\leq Ch^{-s}|z-z_h|_{H^s(A_j'')}
       +Ch^{-\eps}|z|_{H^{2s-\eps}(A_j^{(3)})}.
 \end{aligned}
\]
Here the constants in the inverse estimates are uniform for $0<\eps\leq s/2$.
In the last step, we used the triangle inequality and
\eqref{eqn:local-interp} to estimate $z-I_hz$.
Lemma~\ref{lem:annular-energy-error} and
Lemma~\ref{lem:local-regularity-core} (element sources: Sobolev estimates)
therefore imply
\[
 \begin{aligned}
 |z_h|_{H^{2s}(A_j')}
 &\leq Cr_j^{-d/2}h^{d/2}\ell_h^{\alpha-1}
       +Ch^{-\eps}\eps^{-\xi}(h/r_j)^{d/2}\\
 &\leq Cr_j^{-d/2}h^{d/2}\ell_h^{\alpha-1}.
 \end{aligned}
\]
On the core, the same argument, together with
\eqref{eqn:global-error-bounds} and the global regularity estimate
\eqref{eqn:Sobolev-shift-all}, gives
\[
 \begin{aligned}
 |z_h|_{H^{2s}(A_\ast')}
 &\leq Ch^{-s}|z-z_h|_{H^s(\RR^d)}
       +Ch^{-\eps}|z|_{H^{2s-\eps}(\RR^d)}
 \leq C\ell_h^\xi.
 \end{aligned}
\]
Using $J+1\leq C\ell_h$ and
$H^{d/2}\ell_h^\xi=C_\ast^{d/2}h^{d/2}\ell_h^{\alpha-1}$, we obtain
\begin{equation}\label{eqn:H2s-sum}
 \begin{aligned}
 &\sum_{j=0}^{J-1}r_j^{d/2}|z_h|_{H^{2s}(A_j')}
       +H^{d/2}|z_h|_{H^{2s}(A_\ast')}\\
 &\quad\leq C(J+1)h^{d/2}\ell_h^{\alpha-1}
       +CH^{d/2}\ell_h^\xi
 \leq Ch^{d/2}\ell_h^\alpha.
 \end{aligned}
\end{equation}
Finally, insert \eqref{eqn:H2s-sum} and either \eqref{eqn:L2-sum} or
\eqref{eqn:L2-sum-critical} into \eqref{eqn:discrete-dual-localization}.
This proves \eqref{eqn:discrete-dual-L1} and completes the proof.
\end{proof}

\begin{appendix}
\section{Proof of local regularity estimates when \texorpdfstring{$(s,d)=(1/2,1)$}{(s,d)=(1/2,1)}}\label{app:critical-case}
We prove the critical versions of Lemmas~\ref{lem:local-L2}--\ref{lem:local-regularity-core}.
The $L^2$ estimates follow from the logarithmic Green's function bound.
For the Sobolev estimates, we use the following oscillation estimate to
avoid an additional logarithmic factor in the interior.

\begin{lemma}[Green's function oscillation estimate]\label{lem:critical-green-oscillation}
Let $\Omega\subset\RR$ be a bounded Lipschitz open set, and let $G_\Omega$
be the Green's function of $(-\Delta)^{1/2}$ with zero exterior values.
For $x_0\in\Omega$ and $R>0$ with $\delta(x_0)>8R$, set
$A_R(x_0):=\{z\in\RR:R/2<|z-x_0|<R\}$. Then
\begin{equation}\label{eqn:critical-green-oscillation}
 |G_\Omega(x,y)-G_\Omega(x',y)|
 \leq C\frac{|x-x'|}{R},
 \quad x,x'\in A_R(x_0),\quad |y-x_0|\leq R/4,
\end{equation}
where $C$ is independent of $\Omega$, $x_0$, and $R$.
\end{lemma}
\begin{proof}
The fundamental solution is $\Phi(z):=-\pi^{-1}\log|z|$, with
$(-\Delta)^{1/2}\Phi=\delta_0$ in $\mathcal D'(\RR)$;
see \cite[Theorem~2.3 and formula~(1.21)]{Bucur2016}.
Choose a nonnegative $\eta\in C_c^\infty((-1,1))$ with $\int_\RR\eta=1$.
For $0<\tau<R/8$, set
\[
 \eta_\tau(t):=\tau^{-1}\eta(t/\tau),\qquad
 f_\tau:=\eta_\tau(\cdot-x)-\eta_\tau(\cdot-x'),\qquad
 F_\tau:=\Phi*f_\tau.
\]
Then $f_\tau\in C_c^\infty(\Omega)$ and
$(-\Delta)^{1/2}F_\tau=f_\tau$.
Since $\int_\RR f_\tau=0$, the smooth function $F_\tau$ satisfies
$F_\tau(z)=\mathcal{O}(|z|^{-1})$ and
$F_\tau'(z)=\mathcal{O}(|z|^{-2})$ as $|z|\to\infty$.
Thus $F_\tau\in H^1(\RR)\subset H^{1/2}(\RR)$.

Let $u_\tau\in\widetilde H^{1/2}(\Omega)$ solve
$(-\Delta)^{1/2}u_\tau=f_\tau$ in $\Omega$, and set
$h_\tau:=F_\tau-u_\tau$. Then
\[
 a_{1/2}(h_\tau,v)=0\quad\forall v\in\widetilde H^{1/2}(\Omega),
 \qquad h_\tau=F_\tau\quad\text{in }\Omega^c.
\]
Using the same translate at both poles, we write
\[
 F_\tau(z)=-\frac1\pi\int_\RR\eta_\tau(t)
 \bigl(\log|z-x-t|-\log|z-x'-t|\bigr)\,\mathrm{d}t.
\]
For $z\in\Omega^c\cup\{y\}$ and $|t|\leq\tau$, both distances in the
logarithms are at least $R/8$. The inequality
$|\log a-\log b|\leq|a-b|/\min\{a,b\}$ therefore gives
\begin{equation}\label{eqn:critical-potential-bound}
 \|F_\tau\|_{L^\infty(\Omega^c)}+|F_\tau(y)|
 \leq C\frac{|x-x'|}{R}.
\end{equation}
The weak maximum principle yields
$\|h_\tau\|_{L^\infty(\Omega)}\leq\|F_\tau\|_{L^\infty(\Omega^c)}$.
Indeed, with $M:=\|F_\tau\|_{L^\infty(\Omega^c)}$, testing with
$v:=(h_\tau-M)_+\in\widetilde H^{1/2}(\Omega)$ gives
$0=a_{1/2}(h_\tau,v)\geq a_{1/2}(v,v)$, hence $v=0$;
the lower bound follows by applying the same argument to $-h_\tau$.

The Green representation of $u_\tau$ is continuous near $y$, which lies
away from $\supp f_\tau$. Hence \eqref{eqn:critical-potential-bound} gives
$|u_\tau(y)|\leq C|x-x'|/R$. By continuity of $G_\Omega$ off the diagonal,
\[
 \begin{aligned}
 u_\tau(y)
 &=\int_\RR\eta_\tau(t)
   \bigl(G_\Omega(y,x+t)-G_\Omega(y,x'+t)\bigr)\,\mathrm{d}t\\
 &\longrightarrow G_\Omega(y,x)-G_\Omega(y,x')\qquad(\tau\downarrow0).
 \end{aligned}
\]
The symmetry of $G_\Omega$ completes the proof.
\end{proof}

\begin{proof}[Proof of Lemma~\ref{lem:local-L2} when $(s,d)=(1/2,1)$]
By Theorem~\ref{thm:Green-estimate}, $G_\Omega(x,y)\leq C\log(1+C/|x-y|)$.
For $|k-j|\geq2$, we have $|x-y|\simeq r_{j\wedge k}$ on
$\Omega_k\times\Omega_j$, and hence
\[
 \|w_j\|_{L^2(\Omega_k)}
 \leq C\sqrt{r_k r_j}\bigl(1+|\log r_{j\wedge k}|\bigr)
 \|v_j\|_{L^2(\Omega_j)}.
\]
For $|k-j|\leq1$, Young's inequality with the kernel
$\log(1+C/|\cdot|)\chi_{B_{3r_j}(0)}$ yields
\[
 \|w_j\|_{L^2(\Omega_k)}
 \leq C r_j(1+|\log r_j|)\|v_j\|_{L^2(\Omega_j)}.
\]
Since $r_k\simeq r_j$ in this case, these estimates prove
\eqref{eq:L2-comp-critical}. Squaring and summing over $k\geq J$ gives
\eqref{eq:L2-in-core-critical}.
\end{proof}

\begin{proof}[Proof of Lemma~\ref{lem:local-regularity} when $(s,d)=(1/2,1)$]
If $k\geq j-3$, then $\kappa_{jk}=1$, so the annular estimate follows
directly from the global regularity estimate \eqref{eqn:Sobolev-shift-all}.
The core estimate follows in the same way, since $J>j$ and $\kappa_{jJ}=1$.
It remains to consider $k\leq j-4$, for which
$\kappa_{jk}=(r_j/r_k)^{1/2}$.

\textbf{Case 1: $\delta(x_0)\leq8r_k$.}
For $x\in\Omega_i$ and $y\in\Omega_j$, we have
$\delta(x)\leq Cr_{k\wedge i}$ and $\delta(y)\leq Cr_k$.
If $|i-j|\geq2$, then $|x-y|\simeq r_{i\wedge j}$, and
Theorem~\ref{thm:Green-estimate} gives
\[
 G_\Omega(x,y)
 \leq C\log\left(1+\frac{C\sqrt{r_{k\wedge i}r_k}}{r_{i\wedge j}}\right)
 \leq C\bigl(1+(i-k)_+\bigr),
\]
where $(i-k)_+:=\max\{i-k,0\}$.
For $|i-j|\leq1$, Young's inequality applies with the kernel
$\log(1+Cr_k/|\cdot|)\chi_{B_{3r_j}(0)}$, whose $L^1(\RR)$ norm is
bounded by $Cr_j(1+\log(r_k/r_j))$.
Together these estimates yield
\begin{equation}\label{eqn:critical-boundary-L2}
 \|w_j\|_{L^2(\Omega_i)}
 \leq C\sqrt{r_i r_j}\bigl(1+(i-k)_+\bigr)
       \|v_j\|_{L^2(\Omega_j)},\qquad i\geq0.
\end{equation}
Since $v_j=0$ on $\Omega_k'$, Lemma~\ref{lem:local-sobolev}
(localized Sobolev regularity) and \eqref{eqn:critical-boundary-L2} give
\[
 \begin{aligned}
 |w_j|_{H^{1-\eps}(A_k)}
 &\leq C\eps^{-\xi}r_k^{1/2}
       \sum_{i=0}^\infty r_i^{1/2}\|w_j\|_{L^2(\Omega_i)}
       r_{k\wedge i}^{-2}\\
 &\leq C\eps^{-\xi}\sqrt{r_jr_k}
       \biggl(\sum_{i=0}^k r_i^{-1}
       +r_k^{-2}\sum_{i=k+1}^\infty r_i(1+i-k)\biggr)
       \|v_j\|_{L^2(\Omega_j)}\\
 &\leq C\eps^{-\xi}(r_j/r_k)^{1/2}\|v_j\|_{L^2(\Omega_j)}.
 \end{aligned}
\]
The local $L^2$ term is included in the sum for $|i-k|\leq2$.
The last step uses $\sum_{i=0}^k r_i^{-1}\leq Cr_k^{-1}$ and
$\sum_{i>k}r_i(1+i-k)\leq Cr_k$.

\textbf{Case 2: $\delta(x_0)>8r_k$.}
Here $A_k\subset\Omega$ and $\Omega_j\subset B_{r_k/4}(x_0)$.
Lemma~\ref{lem:critical-green-oscillation} and the Green representation give
\[
 |w_j(x)-w_j(x')|
 \leq Cr_k^{-1}r_j^{1/2}|x-x'|\|v_j\|_{L^2(\Omega_j)},
 \qquad x,x'\in A_k.
\]
Using the normalization in \eqref{eq:slobo}, for $0<t<1$ we obtain
\[
 \begin{aligned}
 |w_j|_{H^t(A_k)}^2
 &\leq Cc_{1,t}r_jr_k^{-2}\|v_j\|_{L^2(\Omega_j)}^2
       \int_{A_k}\int_{A_k}|x-x'|^{1-2t}\,\mathrm{d}x\mathrm{d}x'\\
 &\leq C\frac{c_{1,t}}{1-t}r_jr_k^{1-2t}\|v_j\|_{L^2(\Omega_j)}^2
 \leq Cr_jr_k^{1-2t}\|v_j\|_{L^2(\Omega_j)}^2,
 \end{aligned}
\]
since $c_{1,t}\leq C(1-t)$.
Taking $t=1-\eps$ and using $r_k\leq1$ yields
\[
 |w_j|_{H^{1-\eps}(A_k)}
 \leq Cr_k^\eps(r_j/r_k)^{1/2}\|v_j\|_{L^2(\Omega_j)}
 \leq C\eps^{-\xi}\kappa_{jk}\|v_j\|_{L^2(\Omega_j)},
\]
completing the proof.
\end{proof}

\begin{proof}[Proof of Lemma~\ref{lem:local-L2-core} when $(s,d)=(1/2,1)$]
For $k<J$, we have $|x-y|\simeq r_k$ on $\Omega_k\times T_0$.
Theorem~\ref{thm:Green-estimate} therefore gives
\[
 \|z\|_{L^2(\Omega_k)}
 \leq C r_k^{1/2}(1+|\log r_k|)\|\varphi\|_{L^1(T_0)}
 \leq C\sqrt{h r_k}(1+|\log r_k|)\|\varphi\|_{L^2(T_0)}.
\]
For the core, $|x-y|\leq2r_J$ on $\Omega_\ast\times T_0$.
Young's inequality gives
\[
 \begin{aligned}
 \|z\|_{L^2(\Omega_\ast)}
 &\leq C\bigl\|\log(1+C/|\cdot|)\bigr\|_{L^2(B_{2r_J}(0))}
       \|\varphi\|_{L^1(T_0)}\\
 &\leq C\sqrt{h r_J}(1+|\log r_J|)\|\varphi\|_{L^2(T_0)}.
 \end{aligned}
\]
This proves the assertion.
\end{proof}

\begin{proof}[Proof of Lemma~\ref{lem:local-regularity-core} when $(s,d)=(1/2,1)$]
The dyadic decomposition of $\varphi$ used in the proof for $2s<d$ applies
unchanged. Using the critical version of Lemma~\ref{lem:local-regularity}
in that argument proves \eqref{eqn:local-regularity-core}.
The core estimate \eqref{eqn:local-regularity-core-core} follows from
\eqref{eqn:Sobolev-shift-all}, completing the proof.
\end{proof}

\section{Local energy error estimate up to the boundary}\label{app:local-energy-error}
\begin{proof}[Proof of Theorem~\ref{thm:local-energy-error}]
We follow \cite[Theorem~5.3]{BorthagarayLeykekhmanNochetto2021}, allowing
$B_R$ to intersect $\partial\Omega$, for the full range $0<s<1$.
Quasi-uniformity simplifies the discrete term below.
Let $\eta\in C_0^\infty(B_{3R/4})$ satisfy $0\leq\eta\leq1$,
$\eta=1$ on $B_{R/2}$ and $|\nabla^m\eta|\leq CR^{-m}$ for $m=1,2$.
With $\Pi_h$ denoting nodal interpolation, set
$\psi:=\eta^2u_h-\Pi_h(\eta^2u_h)$.
Since $u_h$ vanishes on $\partial\Omega_h$, we have
$\Pi_h(\eta^2u_h)\in V_h$. Moreover, $h\leq R/16$ gives
$\supp\psi\subset B_{13R/16}$, and its zero extension belongs to
$H_0^1(B_R)$.

\textbf{Step 1: superapproximation.}
All elements on which $\psi$ is nonzero lie in $B_R$.
Summing the elementwise superapproximation estimate
\cite[(5.3)]{BorthagarayLeykekhmanNochetto2021} gives
\begin{equation}\label{eqn:boundary-integer-superapprox}
 \|\psi\|_{L^2(B_R)}+h|\psi|_{H^1(B_R)}
 \leq C\frac{h^2}R\bigl(
       |\eta u_h|_{H^1(B_R)}+R^{-1}\|u_h\|_{L^2(B_R)}\bigr).
\end{equation}
On each element $T$, set $\eta_T:=\eta(x_T)$ for some $x_T\in T$ and
$q_T:=(\eta-\eta_T)u_h$, so that $\eta u_h=\eta_Tu_h+q_T$.
The bounds $\|\eta-\eta_T\|_{L^\infty(T)}\leq Ch R^{-1}$ and
$\|\nabla\eta\|_{L^\infty(T)}\leq CR^{-1}$, the inverse estimate for
$u_h$ and scaling give
\[
 \begin{gathered}
 \|q_T\|_{L^2(T)}+h|q_T|_{H^1(T)}
 \leq ChR^{-1}\|u_h\|_{L^2(T)},\\
 |q_T|_{H^s(T)}
 \leq C\bigl(h^{-s}\|q_T\|_{L^2(T)}+h^{1-s}|q_T|_{H^1(T)}\bigr)
 \leq Ch^{1-s}R^{-1}\|u_h\|_{L^2(T)}.
 \end{gathered}
\]
Applying the inverse estimate to the affine function $\eta_Tu_h$ yields
\[
 \begin{aligned}
 |\eta u_h|_{H^1(T)}
 &\leq Ch^{s-1}\bigl(|\eta u_h|_{H^s(T)}+|q_T|_{H^s(T)}\bigr)
       +|q_T|_{H^1(T)}\\
 &\leq Ch^{s-1}|\eta u_h|_{H^s(T)}+CR^{-1}\|u_h\|_{L^2(T)}.
 \end{aligned}
\]
Since all elements meeting $\supp\eta$ lie in $B_R$ and
$\sum_T|\eta u_h|_{H^s(T)}^2\leq|\eta u_h|_{H^s(B_R)}^2$,
squaring and summing gives
\begin{equation}\label{eqn:boundary-cutoff-inverse}
 |\eta u_h|_{H^1(B_R)}
 \leq Ch^{s-1}|\eta u_h|_{H^s(B_R)}
       +CR^{-1}\|u_h\|_{L^2(B_R)}.
\end{equation}
Since the zero extension of $\psi$ belongs to $H^1(\RR^d)$,
interpolation and Young's inequality yield
\[
 \begin{aligned}
 h^s|\psi|_{H^s(B_R)}
 &\leq C\|\psi\|_{L^2(B_R)}^{1-s}
       \bigl(h|\psi|_{H^1(B_R)}\bigr)^s \leq C\bigl(\|\psi\|_{L^2(B_R)}+h|\psi|_{H^1(B_R)}\bigr).
 \end{aligned}
\]
Together with \eqref{eqn:boundary-integer-superapprox}, this gives
\begin{equation}\label{eqn:boundary-interpolated-superapprox}
 \|\psi\|_{L^2(B_R)}+h^s|\psi|_{H^s(B_R)}
 \leq C\frac{h^2}R\bigl(
       |\eta u_h|_{H^1(B_R)}+R^{-1}\|u_h\|_{L^2(B_R)}\bigr).
\end{equation}
The rough bound $|\eta u_h|_{H^1(B_R)}\leq Ch^{-1}\|u_h\|_{L^2(B_R)}$
and $h\leq R$ give the first estimate below.
Substituting \eqref{eqn:boundary-cutoff-inverse} into
\eqref{eqn:boundary-interpolated-superapprox} and dividing by $h^s$
gives the second:
\begin{equation}\label{eqn:boundary-superapprox}
 \begin{aligned}
 \|\psi\|_{L^2(B_R)}+h^s|\psi|_{H^s(B_R)}
 &\leq C\frac hR\|u_h\|_{L^2(B_R)},\\
 |\psi|_{H^s(B_R)}
 &\leq C\frac hR|\eta u_h|_{H^s(B_R)}
       +C\frac{h^{2-s}}{R^2}\|u_h\|_{L^2(B_R)}.
 \end{aligned}
\end{equation}
This controls the full seminorm on $B_R$, including interactions
across $\partial\Omega_h$.

\textbf{Step 2: localization.}
Galerkin orthogonality gives
\[
 \begin{aligned}
 a_s(u_h,\eta^2u_h)
 &=a_s(u_h,\Pi_h(\eta^2u_h))+a_s(u_h,\psi)\\
 &=a_s(u,\eta^2u_h)-a_s(u,\psi)+a_s(u_h,\psi).
 \end{aligned}
\]
Using \eqref{eqn:boundary-superapprox} and Cauchy--Schwarz to estimate
the cutoff terms as in Steps~2--4 of the proof in \cite{BorthagarayLeykekhmanNochetto2021}, we obtain,
for any $\theta>0$,
\begin{equation}\label{eqn:boundary-energy-reduction}
 \begin{aligned}
 |\eta u_h|_{H^s(B_R)}^2
 &\leq\theta|\eta u_h|_{H^s(B_R)}^2+C_\theta|u|_{H^s(B_R)}^2\\
 &\quad+C_\theta R^{-2s}
       \bigl(\|u\|_{L^2(B_R)}^2+\|u_h\|_{L^2(B_R)}^2\bigr)\\
 &\quad+C_\theta R^{d+2s}\biggl(
       \int_{B_R^c}\frac{|u(y)|+|u_h(y)|}{|y-x_0|^{d+2s}}
       \,\mathrm{d}y\biggr)^2 +|a_s(u_h,\psi)|.
 \end{aligned}
\end{equation}
These estimates use the zero extensions and the separation of the
cutoff supports from $B_R^c$; they do not require $B_R\subset\Omega$.

\textbf{Step 3: the discrete term.}
Since $\psi=0$ outside $B_{13R/16}$, we write
\[
 \begin{aligned}
 a_s(u_h,\psi)
 &=\frac{c_{d,s}}2\int_{B_{7R/8}}\int_{B_{7R/8}}
   \frac{(u_h(x)-u_h(y))(\psi(x)-\psi(y))}{|x-y|^{d+2s}}
   \,\mathrm{d}y\mathrm{d}x\\
 &\quad+c_{d,s}\int_{B_{7R/8}}\int_{B_{7R/8}^c}
   \frac{(u_h(x)-u_h(y))\psi(x)}{|x-y|^{d+2s}}
   \,\mathrm{d}y\mathrm{d}x
 =:J_1+J_2.
 \end{aligned}
\]
The local inverse estimate on $B_{7R/8}\subset B_R$ and
\eqref{eqn:boundary-superapprox} give
\begin{equation}\label{eqn:boundary-near-term}
 \begin{aligned}
 |J_1|
 &\leq Ch^{-s}\|u_h\|_{L^2(B_R)}|\psi|_{H^s(B_R)}\\
 &\leq C\frac{h^{1-s}}R\|u_h\|_{L^2(B_R)}
       |\eta u_h|_{H^s(B_R)}
       +C\frac{h^{2-2s}}{R^2}\|u_h\|_{L^2(B_R)}^2\\
 &\leq\theta|\eta u_h|_{H^s(B_R)}^2
       +C_\theta R^{-2s}\|u_h\|_{L^2(B_R)}^2,
 \end{aligned}
\end{equation}
where the last inequality uses $h\leq R$, $s<1$ and Young's inequality.
For $J_2$, the distance between $\supp\psi$ and $B_{7R/8}^c$ is at
least $R/16$, while $|x-y|\geq |y-x_0|/C$ for
$x\in\supp\psi$ and $y\in B_R^c$. Splitting the $y$-integral at
$\partial B_R$ and using \eqref{eqn:boundary-superapprox}, we obtain
\begin{equation}\label{eqn:boundary-far-term}
 \begin{aligned}
 |J_2|
 &\leq CR^{-2s}\|\psi\|_{L^2(B_R)}\|u_h\|_{L^2(B_R)}
   +CR^{d/2}\|\psi\|_{L^2(B_R)}
       \int_{B_R^c}\frac{|u_h(y)|}{|y-x_0|^{d+2s}}\,\mathrm{d}y\\
 &\leq CR^{-2s}\|u_h\|_{L^2(B_R)}^2
   +CR^{d+2s}\biggl(
       \int_{B_R^c}\frac{|u_h(y)|}{|y-x_0|^{d+2s}}\,\mathrm{d}y\biggr)^2.
 \end{aligned}
\end{equation}
Inserting \eqref{eqn:boundary-near-term} and \eqref{eqn:boundary-far-term}
into \eqref{eqn:boundary-energy-reduction},
choosing $\theta=1/4$ and absorbing the two energy terms gives,
since $\eta=1$ on $B_{R/2}$,
\begin{equation}\label{eqn:boundary-discrete-energy}
 \begin{aligned}
 |u_h|_{H^s(B_{R/2})}^2
 &\leq C|u|_{H^s(B_R)}^2
       +CR^{-2s}\bigl(\|u\|_{L^2(B_R)}^2+\|u_h\|_{L^2(B_R)}^2\bigr)\\
 &\quad+CR^{d+2s}\biggl(
       \int_{B_R^c}\frac{|u(y)|+|u_h(y)|}{|y-x_0|^{d+2s}}
       \,\mathrm{d}y\biggr)^2.
 \end{aligned}
\end{equation}
For any $v_h\in V_h$, we have $R_h^s(u-v_h)=u_h-v_h$.
Apply \eqref{eqn:boundary-discrete-energy} with $(u,u_h)$ replaced by
$(u-v_h,u_h-v_h)$.
Using $u_h-v_h=(u-v_h)-(u-u_h)$ and the triangle inequality then
gives \eqref{eqn:local-energy-error-split}, since all functions vanish
outside $\Omega$.
\end{proof}

\end{appendix}

\bibliographystyle{amsplain}
\bibliography{frac_linf_lesshalf}   

\end{document}